\documentclass[11pt]{article}
\usepackage[english]{babel}
\usepackage[a4paper,scale=0.75,hcentering]{geometry}
\usepackage[hyperfootnotes=false]{hyperref} 
\usepackage{setspace}
\hypersetup{
	colorlinks = true,
	linkcolor = {blue},
	urlcolor = {red},
	citecolor = {blue}
}
\usepackage{amsmath, amsthm, amssymb, mathrsfs, mathtools, enumitem}
\usepackage[x11names]{xcolor}

\allowdisplaybreaks

\newcounter{results}[section] 

\theoremstyle{plain}
\newtheorem{theorem}[results]{Theorem}
\newtheorem{lemma}[results]{Lemma}

\newtheorem{corollary}[results]{Corollary}

\newtheorem*{theorem*}{Theorem}
\newtheorem*{lemma*}{Lemma}
\newtheorem*{proposition*}{Proposition}
\newtheorem*{corollary*}{Corollary}
\newtheorem*{exercise*}{Exercise}
\newtheorem*{fact*}{Fact}
\newtheorem*{problem*}{Problem}
\newtheorem*{conjecture*}{Conjecture}

\theoremstyle{remark}
\newtheorem{remark}[results]{Remark}

\newtheorem*{remark*}{Remark}
\newtheorem*{question*}{Question}

\theoremstyle{definition}
\newtheorem{definition}[results]{Definition}

\newtheorem*{definition*}{Definition}
\newtheorem*{example*}{Example}

\numberwithin{equation}{section}

\newcommand{\N}{\ensuremath{\mathbb N}} 
\newcommand{\R}{\ensuremath{\mathbb R}} 
\providecommand{\C}{}
\renewcommand{\C}{\ensuremath{\mathcal C}} 
\renewcommand{\S}{\ensuremath{\mathbb S}} 

\newcommand{\LL}{\ensuremath{\mathcal L}}
\newcommand{\M}{\ensuremath{\mathcal M}}
\makeatletter 
\DeclarePairedDelimiter{\@tmpabs}{\lvert}{\rvert}
\newcommand{\@absstar}[1]{{\@tmpabs*{#1}}}
\newcommand{\@absnostar}[2][]{{\@tmpabs[#1]{#2}}}
\newcommand{\abs}{\@ifstar\@absstar\@absnostar}
\makeatother

\makeatletter 
\DeclarePairedDelimiter{\@tmpnorm}{\lVert}{\rVert}
\newcommand{\@normstar}[1]{{\@tmpnorm*{#1}}}
\newcommand{\@normnostar}[2][]{{\@tmpnorm[#1]{#2}}}
\newcommand{\norm}{\@ifstar\@normstar\@normnostar}
\makeatother

\let\div\undefined
\DeclareMathOperator{\div}{div}

\ifdefined\comma
	\renewcommand{\comma}{\ensuremath{\, \text{, }}}
\else
	\newcommand{\comma}{\ensuremath{\, \text{, }}}
\fi

\definecolor{removed}{RGB}{255, 0, 0}
\definecolor{added}{RGB}{0,192,0}

\usepackage[hyperfootnotes=false]{hyperref} 
\usepackage{setspace}
\hypersetup{
	colorlinks = true,
	linkcolor = {blue},
	urlcolor = {red},
	citecolor = {blue}
}

\usepackage[nameinlink,capitalise,sort]{cleveref} 
\crefname{equation}{}{} 
\crefname{enumi}{}{} 

\allowdisplaybreaks

\numberwithin{equation}{section}
\begin{document}

\title{\bf The stability on the Caffarelli-Kohn-Nirenberg and Hardy-type inequalities and beyond \thanks{This work is supported by National Key R\&D Program of China (Grant 2023YFA1010001) and  NSFC(12171265). E-mail addresses: zhou-yx22@mails.tsinghua.edu.cn(Zhou),    zou-wm@mail.tsinghua.edu.cn (Zou)}}

\author{{\bf Yuxuan Zhou, Wenming Zou}\\ {\footnotesize \it  Department of Mathematical Sciences, Tsinghua University, Beijing 100084, China.} }

\date{}



\maketitle

\begin{abstract}
{\small In this paper, we establish several improved Caffarelli-Kohn-Nirenberg and Hardy-type inequalities. Our main results are divided into two parts.

\vskip0.1in
In the first part, we consider the following Caffarelli-Kohn-Nirenberg inequality:
\begin{equation*}
    \left(\int_{\R^n}|x|^{-pa}|\nabla u|^pdx\right)^{\frac{1}{p}}\geq S(p,a,b)\left(\int_{\R^n}|x|^{-qb}|u|^qdx\right)^{\frac{1}{q}},\quad\forall\; u\in D_a^p(\R^n),
\end{equation*}
where $S(p,a,b)$ is the sharp constant and $a,b,p,q$ satisfy the relations:
\begin{equation*}
    1<p<n,\quad 0\leq a<\frac{n-p}{p},\quad a\leq b<a+1,\quad q=\frac{np}{n-p(1+a-b)}.
\end{equation*}
We establish gradient stability of this inequality in both functional and critical settings, and we derive some functional properties of the stability constant. Building on the gradient stability, we also obtain several refined Sobolev-type embeddings involving weak Lebesgue norms for functions supported in general domains.
\vskip0.1in
In the second part, we focus on various classical Hardy-type inequalities, including the standard Hardy inequality, the $L^p$-logarithmic Sobolev inequality with weights, the logarithmic Hardy inequality, the Hardy-Morrey inequality, the Hardy-Sobolev interpolation inequality, and the interpolated Caffarelli-Kohn-Nirenberg inequality. We investigate their weighted versions and derive corresponding extremal functions, refinements, new remaining  terms and stability constants.
\vskip0.1in
An important tool in our arguments is a class of simple yet super useful transformations. These transformations enable us to reduce inequalities with complicated weights to simpler forms and to establish explicit relations between stability constants.


}
\end{abstract}
\newpage
\tableofcontents

\section{Introduction}\label{sec1}
\subsection{Overview of this paper}\label{sec1.1}
The main subject of this paper is to establish several improved Caffarelli-Kohn-Nirenberg and Hardy type inequalities. Our main results are organized into two parts. In the first part, we address the gradient stability of the Caffarelli-Kohn-Nirenberg inequality and derive several new inferences. In the second part, we examine various Hardy-type inequalities, establishing their weighted versions and deriving corresponding extremal functions, refinements, new remaining  terms and stability constants.

\vskip0.1in
In Subsection \ref{sec1.2}, we provide a detailed introduction to the history of the Caffarelli-Kohn-Nirenberg inequality. We survey previous works, state our main theorems, and offer several remarks. The proofs of these theorems are presented in Section \ref{sec2}.

\vskip0.1in
In Subsection \ref{sec1.3}, to enhance the readability of this paper and avoid making the introduction too lengthy, we provide only a brief overview of our results on various Hardy type inequalities, with a more detailed introduction available in Section \ref{sec3}. We divide Section \ref{sec3} into six subsections, each dedicated to a specific type of Hardy  type inequality, where we state our main theorems and present the proofs. Except for the proofs, Section \ref{sec3} is self-contained, allowing interested readers to engage with it directly.

\vskip0.1in
An important tool we utilize is a class of simple yet effective transformations. A prototype of these transformations is introduced at the beginning of Subsection \ref{sec2.1}, derived from Horiuchi's work \cite{Hor}. Some variants of this transformation are discussed in Section \ref{sec3}. We note that similar tools have been employed in several works (for example, \cite{Don,Lam0}) to establish new inequalities along with corresponding sharp constants and extremizers. In this work, we not only apply these transformations to investigate a broad class of Caffarelli-Kohn-Nirenberg and Hardy-type inequalities but also use them to derive stability results and explicit stability constants.

\subsection{About the Caffarelli-Kohn-Nirenberg inequality}\label{sec1.2}
Here we are concerned with the following Caffarelli-Kohn-Nirenberg inequality proved in \cite{Caf}:
\begin{equation}\label{ckn}
    \left(\int_{\R^n}|x|^{-pa}|\nabla u|^pdx\right)^{\frac{1}{p}}\geq S(p,a,b)\left(\int_{\R^n}|x|^{-qb}|u|^qdx\right)^{\frac{1}{q}},\quad\forall\; u\in D_a^p(\R^n),
\end{equation}
where $S(p,a,b)$ is the sharp constant and $a,b,p,q$ satisfy the relations:
\begin{equation}\label{ckn1}
    1<p<n,\quad 0\leq a<\frac{n-p}{p},\quad a\leq b<a+1,\quad q=\frac{np}{n-p(1+a-b)}.
\end{equation}
The space $D_a^p(\R^n)$ is the completion of $C_c^{\infty}(\R^n)$ with respect to the norm
\begin{equation*}
    \norm*{u}_{D_a^p}:=\left(\int_{\R^n}|x|^{-pa}|\nabla u|^pdx\right)^{\frac{1}{p}}.
\end{equation*}

When $a=b=0$, \eqref{ckn} reduces to the classical Sobolev inequality. It is well-known that Aubin \cite{Aub} and Talenti \cite{Tal} found the explicit sharp constant $S(p,0,0)$ and determined all extremal functions for this case. Later, relying on symmetrization techniques, Lieb \cite{Lie} extended their results to the case $p=2,a=0$. For the case $p=2,(a,b)\neq (0,0)$, the first complete classification result was given by Chou and Chu \cite{Cho}. As for the case $p\neq 2$ and $(a,b)\neq (0,0)$, based on a simple transformation, Horiuchi \cite{Hor} showed that
\begin{equation}\label{ckn2}
\begin{aligned}
     S(p,a,b)=&\;n^{\frac{1}{p}}(p-1)^{\frac{\gamma}{n}-1+\frac{1}{p}}(n-\gamma p)^{\frac{\gamma}{n}-\frac{1}{p}}(n-p-pa)^{1-\frac{\gamma}{n}}\left(\frac{2\pi^{\frac{n}{2}}}{p\gamma}\right)^{\frac{\gamma}{n}}\\
     &\times \left(\frac{\Gamma\left(\frac{n}{\gamma p}\right)\Gamma\left(\frac{n(p-1)}{\gamma p}\right)}{\Gamma\left(\frac{n}{2}\right)\Gamma\left(\frac{n}{\gamma}\right)}\right)^{\frac{\gamma}{n}}
\end{aligned}
\end{equation}
and the equality in \eqref{ckn} can be achieved by radial functions taking the form:
\begin{equation}\label{ckn3}
    \frac{A}{\Big(1+B|x|^{\frac{p\gamma(n-p-pa)}{(p-1)(n-p\gamma)}}\Big)^{\frac{n}{p\gamma}-1}},
\end{equation}
where $\gamma=1+a-b,A\in \R$ and $B>0$. Assuming $a=0$, Ghoussoub and Yuan \cite{Gho} proved that \eqref{ckn3} actually gave all extremal functions for \eqref{ckn}. Later, Caldiroli and Musina \cite{Cal} studied the symmetry-breaking phenomenon and showed all extremal functions are radially symmetric. Subsequently, Musina \cite{Mus} showed that every radially symmetric extermal function must have the form \eqref{ckn3}. Recently, Ciraolo and Corso \cite{Cir} found that, in some special cases, \eqref{ckn3} gives all positive critical points for the corresponding Euler-Lagrange equation. Their proof relies on some delicate analysis on a suitable Riemannian manifold.

Let us present a summarization of all the aforementioned conclusions in the following theorem and give the definition of the manifold $\M_{p,a,b}$ formed by extremal functions:
\begin{theorem}\label{thm1}
    Assume $a,b,p,q$ satisfy the relations in \eqref{ckn1}. Then the equality in \eqref{ckn} holds if and only if $u$ takes the form specified in \eqref{ckn3} (if $a=b=0$, then up to a translation). We denote the manifold formed by extremal functions as $\mathcal{M}_{p,a,b}$:
    \begin{equation*}
        \mathcal{M}_{p,a,b}:=\left\{A\Big(1+B|x|^{\frac{p\gamma(n-p-pa)}{(p-1)(n-p\gamma)}}\Big)^{1-\frac{n}{p\gamma}}\Big|A\in\R,B>0,\gamma=1+a-b\right\},\quad(a,b)\neq (0,0),
    \end{equation*}
    and for the case $(a,b)=(0,0)$:
    \begin{equation*}
        \mathcal{M}_{p,0,0}:=\left\{A\Big(1+B|x-x_0|^{\frac{p}{p-1}}\Big)^{1-\frac{n}{p}}\Big|A\in\R,B>0,x_0\in\R^n\right\}.
    \end{equation*}
\end{theorem}
For the reader's convenience, in Subsection \ref{sec2.1}, we present a straightforward proof for the case $a>0$ to highlight a simple yet clever tranformation (refer to the beginning of Subsection \ref{sec2.1}). This transformation, along with some of its variants, proves to be useful for our main results in this paper. Additionally, we observe an interesting phenomenon indicated by this proof: if the classification of radial extremizers for the case $a=0$ is established, we can then characterize all extremizers in the case $a>0$. Similar phenomena also appear in Section \ref{sec3}.

\vskip0.1in
Our first main result concerns the gradient stability for \eqref{ckn}. The question on the stability of functional inequalities was initially raised by Br\'ezis and Lieb \cite{Bre}. They obtained a sharpened Sobolev embedding on bounded domains and posed the question of whether similar results hold for homogeneous Sobolev inequalities in $\R^n$. This question was answered affirmatively by Bianchi and Egnell \cite{Bia} in the case $p=2$. They showed that for any function $u\in D_0^2(\R^n)$, the gap in the Sobolev inequality can control the $D_0^2$-distance between $u$ and an extremizer. Subsequently, their results were extended to high-order and fractional-order Hardy-Sobolev inequalities. Related works can be found in \cite{Bar,Che,Lu,Rad}. Assuming $p=2$, Wang and Willem \cite{Wan} were the first to establish gradient stability in the full parameter domain \eqref{ckn1}. For the case $a<0,p=2$, we refer to \cite{Fra,Wei} for some related works.

 \vskip0.1in

 For the case $p\neq 2$, due to the absence of the Hilbertian structure, more delicate arguments are required to handle the stability. We refer to \cite{Cia,Fig,Neu} and references therein for various previous works and partial results. The first sharp stability result was given by Figalli and Zhang \cite{Fig1}. Relying on refined spectral estimates and compactness arguments, they proved that, for any $1<p<n$, there exists a positive constant $K$ such that
\begin{equation}\label{asdf}
    \frac{\norm*{\nabla u}_{L^p}}{\norm*{u}_{L^{\frac{np}{n-p}}}}-S(p,0,0)\geq K\inf_{v\in\mathcal{M}_{p,0,0}}\left(\frac{\norm*{\nabla u-\nabla v}_{L^p}}{\norm*{\nabla u}_{L^p}}\right)^{\max\{2,p\}},\quad \forall\; u\in D_0^p(\R^n)\backslash \{0\}.
\end{equation}
The exponent $\max\{2,p\}$ is optimal. Later, by adapting the methods developed in \cite{Fig1}, Deng and Tian \cite{Den,Den1} extended \eqref{asdf} to the weighted inequality \eqref{ckn}. In \cite{Den}, they established stability estimates across the full parameter range, under the assumption that $u$ is radially symmetric. In a subsequent work \cite{Den1}, they focused on the special case $a=0$ and proved the following estimate:
\begin{equation}\label{DT}
    \frac{\norm*{\nabla u}_{L^p}}{\norm*{|x|^{-b}u}_{L^q}}-S(p,0,b)\geq K\inf_{v\in\mathcal{M}_{p,0,b}}\left(\frac{\norm*{\nabla u-\nabla v}_{L^p}}{\norm*{\nabla u}_{L^p}}\right)^{\max\{2,p\}},\quad \forall\; u\in D_0^p(\R^n),
\end{equation}
where $K$ is a universal positive constant, and the exponent $\max\{2,p\}$ is optimal.

\vskip0.1in
In this paper, our first main result is to establish gradient stability for \eqref{ckn} in the full parameter region \eqref{ckn1}. Specifically, we prove:
\begin{theorem}\label{thm2}
    Assume $a,b,p,q$ satisfy the relations in \eqref{ckn1} and $0<a<b$. Then there exists a positive constant $K$ such that for any nontrivial function $u\in D_a^p(\R^n)$, we have
    \begin{equation}\label{sta1}
         \frac{\norm*{|x|^{-a}\nabla u}_{L^p}}{\norm*{|x|^{-b}u}_{L^q}}-S(p,a,b)\geq K\inf_{v\in\mathcal{M}_{p,a,b}}\left(\frac{\norm*{|x|^{-a}(\nabla u-\nabla v)}_{L^p}}{\norm*{|x|^{-a}\nabla u}_{L^p}}\right)^{\max\{2,p\}}.
    \end{equation}
    Recall that $\mathcal{M}_{p,a,b}$ denotes the extremal manifold. Moreover, the exponent $\max\{2,p\}$ is optimal.
\end{theorem}

\begin{theorem}\label{thm3}
     Assume $a,b,p,q$ satisfy the relations in \eqref{ckn1} and $0<a=b$. Then there exists a positive constant $K$ such that for any nontrivial function $u\in D_a^p(\R^n)$, we have
    \begin{equation}\label{sta2}
         \frac{\norm*{|x|^{-a}\nabla u}_{L^p}}{\norm*{|x|^{-a}u}_{L^q}}-S(p,a,a)\geq K\inf_{v\in\mathcal{M}_{p,a,a}}\left(\frac{\norm*{|x|^{-a}(\nabla u-\nabla v)}_{L^p}}{\norm*{|x|^{-a}\nabla u}_{L^p}}\right)^{\max\{4,2p\}}.
    \end{equation}
    If we assume that $u$ is centrally symmetric about the origin (i.e., $u(x)=u(-x)$ for any $x\in\R^n$), the exponent $\max\{4,2p\}$ can be replaced by $\max\{2,p\}$.
\end{theorem}

\begin{remark}
    The optimality of the exponent $\max\{2,p\}$ in \eqref{sta1} can be derived by selecting appropriate test functions. We omit the details here, as the arguments are essentially the same as those presented in \cite[Remark 1.5]{Den} and \cite[Remark 1.2]{Fig1}.
\end{remark}
\begin{remark}
    The main idea behind the derivation of the two theorems is to reduce the original inequalities to simpler forms. A key tool in this process is the transformation introduced at the beginning of Subsection~\ref{sec2.1}. When \( 0 < a < b \), the reduction proceeds smoothly. However, in the case \( 0 < a = b \), difficulties arise due to the fact that \( \mathcal{M}_{p,0,0} \) is translation-invariant, whereas \( \mathcal{M}_{p,a,a} \) is not. This discrepancy necessitates careful estimates during the reduction process, which leads to a loss in the exponent. We conjecture that, even when \( 0 < a = b \), the estimate \eqref{sta2} still holds for general functions, provided the exponent \( \max\{4,2p\} \) is replaced by \( \max\{2,p\} \).
\end{remark}

\begin{remark}
    We believe that generalizing the methods developed in \cite{Den,Den1,Fig1} provides another promising approach to establishing quantitative stability for inequality \eqref{ckn} across the full parameter range. In \cite{Den}, Deng and Tian obtained stability results for the entire parameter region, but only under the assumption of radial symmetry. The reason for this restriction is that, in the radially symmetric setting, the derivation of spectral estimates (see \cite[Propositions~3.2 and~3.4]{Den}) is significantly simplified. In \cite{Den1}, by adapting arguments from \cite[Propositions~3.2 and~3.6]{Fig1}, Deng and Tian successfully derived the necessary spectral estimates (see \cite[Propositions~1.2 and~4.2]{Den1}) for general functions in the special case \( a = 0 \), and thus established stability in that case. Therefore, once appropriate spectral estimates are established for the full parameter range, one can reasonably expect to obtain corresponding stability results as well.
\end{remark}

Based on the stability results \eqref{DT}, \eqref{sta1} and \eqref{sta2}, we derive several improved Sobolev-type embeddings involving weak Lebesgue norms over general domains. When restricted to domains with finite Lebesgue measures, we have:
\begin{corollary}\label{ccc1}
    Assume $a,b,p,q$ satisfy the relations in \eqref{ckn1}. Set $p_1=\frac{n(p-1)}{n-p-a}$,  $p_2=\frac{n(p-1)}{n-a-1}$, and $p_3=\frac{n-p-pa}{np(p-1)}$. Let $\Omega\subset\R^n$ be a  domain with $|\Omega|<\infty$. Then there exists a positive constant $\Bar{K}$, depending only on $a,b,p,q,n$, such that for any $u\in D_a^p(\Omega)$, we have
    \begin{equation}\label{ccc2}
        \norm*{|x|^{-a}\nabla u}_{L^p(\Omega)}^\alpha-S(p,a,b)^\alpha\norm*{|x|^{-b}u}^\alpha_{L^q(\Omega)}\geq \Bar{K}|\Omega|^{-\alpha p_3}\norm*{|x|^{-a}u}_{L_w^{p_1}(\Omega)}^\alpha
    \end{equation}
    and
    \begin{equation}\label{ccc3}
        \norm*{|x|^{-a}\nabla u}_{L^p(\Omega)}^\alpha-S(p,a,b)^\alpha\norm*{|x|^{-b}u}^\alpha_{L^q(\Omega)}\geq\Bar{K}|\Omega|^{-\alpha p_3}\norm*{|x|^{-a}\nabla u}_{L_w^{p_2}(\Omega)}^\alpha.
    \end{equation}
    Here, $\alpha$ takes the value $\max\{4,2p\}$ in the case $p\neq 2$ and $0<a=b$, and $\max\{2,p\}$ in other cases. The space $D_a^p(\Omega)$ denotes the closure of $C_c^\infty(\Omega)$ in $D_a^p(\R^n)$. The weak Lebesgue norm is defined by
    \begin{equation*}
        \norm*{f}_{L_w^p(\Omega)}:=\sup_{t>0}\,t\,|\{x\in\Omega,|f(x)|>t\}|^{\frac{1}{p}}.
    \end{equation*}
\end{corollary}
\begin{remark}
    The method of using gradient stability to obtain strengthened Sobolev-type embeddings with weak Lebesgue norms was first introduced by Bianchi and Egnell \cite{Bia}. By combining gradient stability with a compactness argument, they established the following improved Sobolev embedding:
    \begin{equation*}
        \norm*{\nabla u}_{L^2(\Omega)}^2-S(2,0,0)^2\norm*{u}^2_{L^{\frac{2n}{n-2}}(\Omega)}\geq \Bar{K}(\Omega)\norm*{u}_{L_w^{\frac{n}{n-2}}(\Omega)}^2
    \end{equation*}
    for any bounded domain $\Omega\subset\R^n$. Here, $\Bar{K}(\Omega)$ is a positive constant that depends only on $n$ and $\Omega$. Subsequently, R\v{a}dulescu, Smets and Willem \cite{Rad}, as well as Wang and Willem \cite{Wan}, applied similar methods to derive improved versions of the Caffarelli-Kohn-Nirenberg inequalities \eqref{ckn} with parameters $p=2,a=0,b>0$ and $p=2,a>0,b>0$, respectively. This method can also be utilized to address other Sobolev-type inequalities. We refer to \cite{Che,Gaz,Gaz1,Zha} and the references therein for further discussions. It is worth mentioning that an improvement on this approach was first provided by Chen, Frank and Weth \cite{Che}. By replacing the compactness argument with careful estimates, they established sharpened fractional Sobolev embeddings over domains with finite measures. Moreover, their estimates yield an explicit dependence of $\Bar{K}(\Omega)$ on the measure of $\Omega$. In this paper, we adapt the ideas from \cite{Bia,Che} to deduce \eqref{ccc2} and \eqref{ccc3}.
\end{remark}
When considering domains with infinite measures, we need to make some assumptions. For a domain $\Omega\subset\mathbb{R}^n$, the principal $p$-Laplacian eigenvalue is defined by
\begin{equation}\label{ass}
    \lambda(p,\Omega)=\inf_{D^p_0(\Omega)}\frac{\int_{\Omega}|\nabla u|^pdx}{\int_{\Omega}|u|^pdx}.
\end{equation}
We say that $\Omega$ satisfies the $(A_0)$ condition if there exist an open cone $V$ with vertex at 0 and a number $R>0$ such that $\Omega^c\supset V\backslash B_R(0)$. We say that $\Omega$ satisfies the $(A_1)$ condition if there exist an open cone $V$ with vertex at 0 and a number $R>0$ such that for all $y\in\Omega$, it holds that $\Omega^c\supset (y+V)\backslash B_R(y)$. It is easy to see that strips or subdomains of strips satisfy the $(A_1)$ condition and have positive principal $p$-Laplacian eigenvalues. Here, by strips, we mean domains that are bounded in some direction. Our conclusions are as follows:
\begin{corollary}\label{ccc6}
    Assume $a=b=0$ and that $p,q$ satisfy the relations in \eqref{ckn1}. Set $p_1=\frac{n(p-1)}{n-p}$. Assume $\Omega\subset\R^n$ satisfies the $(A_1)$ condition and that $\lambda(p,\Omega)>0$. Then, if $\sqrt{n}\leq p<n$, there exists a positive constant $\Bar{K}$, depending only on $p,q,n,R,|V\cap \S^{n-1}|$, such that for any $u\in D_0^p(\Omega)$, we have
    \begin{equation}\label{ccc7}
        \norm*{\nabla u}_{L^p(\Omega)}^{\max\{2,p\}}-S(p,0,0)^{\max\{2,p\}}\norm*{u}^{\max\{2,p\}}_{L^q(\Omega)}\geq \Bar{K}\min\{1,\lambda(p,\Omega)^{\frac{\beta\max\{2,p\}}{p}}\}\norm*{u}_{L_w^{p_1}(\Omega)}^{\max\{2,p\}}.
    \end{equation}
    Here, $\beta$ is a nonnegative number such that
    \begin{equation}\label{qwe}
        \frac{1}{p_1}=\frac{\beta}{p}+\frac{1-\beta}{q}=\frac{\beta}{p}+\frac{(1-\beta)(n-p)}{np}.
    \end{equation}
\end{corollary}
\begin{corollary}\label{ccc8}
    Assume $a,b,p,q$ satisfy the relations in \eqref{ckn1}. Set $p_1=\frac{n(p-1)}{n-p-a}$. Assume $\Omega\subset\R^n$ satisfies the $(A_0)$ condition and that $\lambda(p,\Omega)>0$. Then, if $(a,b)\neq (0,0)$ and $a\geq\frac{n-p^2}{p}$, there exists a positive constant $\Bar{K}$, depending only on $a,b,p,q,n,R,|V\cap \S^{n-1}|$, such that for any $u\in D_a^p(\Omega)$, we have
    \begin{equation}\label{ccc9}
        \norm*{|x|^{-a}\nabla u}_{L^p(\Omega)}^{\alpha}-S(p,a,b)^{\alpha}\norm*{|x|^{-b}u}^\alpha_{L^q(\Omega)}\geq \Bar{K}\min\{1,\lambda(p,\Omega)^{\frac{\beta\alpha}{p}}\}\norm*{|x|^{-a}u}_{L_w^{p_1}(\Omega)}^{\alpha}.
    \end{equation}
    Here, $\alpha$ takes the value $\max\{4,2p\}$ in the case $p\neq 2$ and $0<a=b$, and $\max\{2,p\}$ in other cases. The parameter $\beta$ is a nonnegative number satisfying \eqref{qwe}.
\end{corollary}
\begin{remark}
    The conditions $(A_0)$ and $(A_1)$, along with the above two corollaries, are motivated by the work of Wang and Willem \cite[Theorem 4, Theorem 5]{Wan}. They proved \eqref{ccc7} and \eqref{ccc9} in the case $p=2$. Their proof relies on gradient stability and a compactness argument. However, it is important to note that the compactness argument does not allow for an explicit dependence of the stability constant $\Bar{K}$ on the domain $\Omega$. In this paper, by replacing the compactness argument with careful estimates motivated by \cite{Che}, we establish \eqref{ccc7} and \eqref{ccc9} for the general case $1<p<n$. It turns out that the constant $\Bar{K}$ only depends on $a,b,p,q,n,R,|V\cap \S^{n-1}|$, and $\lambda(p,\Omega)$.
\end{remark}
\begin{remark}
    When $p=2$, Wei and Wu \cite{Wei}, as well as Frank and Peteranderl \cite{Fra}, established gradient stabilities for the Caffarelli-Kohn-Nirenberg inequality \eqref{ckn} in the case where $a<0$ and $b\geq b_{\text{FS}}(a)$. Here, $b_{\text{FS}}$ denotes the Felli-Schneider curve (see \cite{Fel}). When $a<0$ and $b> b_{\text{FS}}(a)$, we can also  derive estimates analogous to \eqref{ccc2}, \eqref{ccc3}, \eqref{ccc7}, and \eqref{ccc9}.
\end{remark}

\vskip0.2in
Our second main result pertains to the functional properties of the stability constant $K$ appearing in \eqref{sta1}. Assuming that $a,b,p,q$ satisfy the relations in \eqref{ckn1}, we define $K(p,a,b)$ as the optimal constant such that the following estimate holds for any function $u\in D_a^p(\R^n)\backslash\{0\}$:
\begin{equation*}
    \frac{\norm*{|x|^{-a}\nabla u}_{L^p}}{\norm*{|x|^{-b}u}_{L^q}}-S(p,a,b)\geq K(p,a,b)\inf_{v\in\mathcal{M}_{p,a,b}}\left(\frac{\norm*{|x|^{-a}(\nabla u-\nabla v)}_{L^p}}{\norm*{|x|^{-a}\nabla u}_{L^p}}\right)^{\max\{2,p\}}.
\end{equation*}
It is noteworthy that, based on the results from \cite{Fig1,Den1,Wei}, as well as our results Theorem \ref{thm2} and Theorem \ref{thm3}, $K(p,a,b)$ is positive when $a\neq b$, when $a=b=0$, or when $p=2,a=b>0$. However, in the case $p\neq2$ and $a=b>0$, our stability results do not show whether $K(p,a,a)$ is positive, although we conjecture that it is indeed positive.

\vskip0.1in
Finding explicit upper and lower bounds for the stability constant $K(p,a,b)$ is a newly developed problem. Recently, Dolbeault, Esteban, Figalli, Frank and Loss \cite{Dol} obtained an explicit lower bound for $K(2,0,0)$ using techniques involving competing symmetry and continuous Steiner symmetrization. Subsequently, based on a clever third-order expansion, K\"onig \cite{kon,kon1} gave an explicit upper bound for $K(2,0,0)$ and showed that $K(2,0,0)$ can be attained by some function $u\in D_0^2(\R^n)\backslash \M_{2,0,0}$. In a parallel vein, Wei and Wu \cite{Wei1} as well as Deng and Tian \cite{Den2} found explicit upper bounds for $K(2,a,b)$ and demonstrated the attainability. We refer to \cite{Chen,Zha} for analogous results on the fractional Sobolev and Sobolev trace inequalities.

\vskip0.1in
In this paper, our second main result focus on functional properties of $K(p,a,b)$. Specifically, we view $K(p,a,b)$ as a function of $p,a,b$ and investigate its properties. Our result appears to be the first in this direction.
\begin{theorem}\label{thm4}
     $(i)$ Assume $(a_l,b_l,p_l,q_l)_{l\in\N^+}$ and $(a,b,p,q)$ satisfy the relations in \eqref{ckn1}. If
     \begin{align}
         (a_l,b_l,p_l,q_l)\rightarrow (a,b,p,q)\quad \text{as }l\rightarrow +\infty,\nonumber
     \end{align}
     then we have
     \begin{equation}\label{sta3}
         \limsup_{l\rightarrow +\infty}\;K(p_l,a_l,b_l)\leq K(p,a,b).
     \end{equation}
    $(ii)$ Assume $a_1,b_1,p,q_1$ and $a_2,b_2,p,q_2$ satisfy the relations in \eqref{ckn1}. If $0<b_1-a_1=b_2-a_2$ with $a_1<a_2$, or $0<a_1=b_1<a_2=b_2$, then we have
     \begin{equation}\label{sta4}
         \frac{K(p,a_2,b_2)}{(n-p-a_2p)^{\nu}}\geq\frac{K(p,a_1,b_1)}{(n-p-a_1p)^{\nu}}\;,
     \end{equation}
     where $\nu=1+\max\{1,p-1\}\frac{\gamma}{n}$ and $\gamma=1+a_1-b_1=1+a_2-b_2$. Moreover, if $K(p,a_2,b_2)$ can be attained by some function $u\in D_a^p(\R^n)\backslash\M_{p,a_2,b_2}$, then the inequality in \eqref{sta4} is strict.
\end{theorem}
\begin{remark}
    When $p=2$, $a<0$, and $b>b_{\text{FS}}(a)$, the gradient stability results established in \cite{Wei} allow us to employ similar arguments to demonstrate that \eqref{sta3} and \eqref{sta4} still hold. However, we are uncertain whether $K(p,0,0)$ is comparable to $K(p,a,a)$ for $a>0$.
\end{remark}
\begin{remark}
    There are two open questions that remain. The first question is whether $K$ is a continuous function of $a,b,p$. The second question is whether there exists a number $\nu_0$ such that
    \begin{equation}\label{dfg}
        \frac{K(p,a_2,b_2)}{(n-p-a_2p)^{\nu_0}}=\frac{K(p,a_1,b_1)}{(n-p-a_1p)^{\nu_0}}\;.
    \end{equation}
    Regarding the first question, \eqref{sta3} indicates that $K(p,a,b)$ is upper semicontinuous. For the second question, there are some partial results available when $p=2$. From \cite{Wei1}, we know that \eqref{sta4} is always strict. Additionally, from \cite{Fra}, we find that $K(2,a,b)$ goes to zero as $(a,b)$ approaches the Felli-Schneider curve, suggesting that \eqref{dfg} cannot always hold.
\end{remark}

\vskip0.1in
Our third main result relates to the stability of \eqref{ckn} in the critical point setting. Consider the corresponding Euler-Lagrange equation:
\begin{equation}\label{el}
    H(u):=\div(|x|^{-pa}|\nabla u|^{p-2}\nabla u)+|x|^{-qb}|u|^{q-2}u=0.
\end{equation}
We define $\mathcal{M}_{p,a,b}^s$ as a subset of $\mathcal{M}_{p,a,b}$ that contains all positive solutions of the form \eqref{ckn3} of the equation \eqref{el} (When $a=b=0$, translations are allowed):
\begin{equation*}
        \mathcal{M}^s_{p,a,b}:=\left\{A_{p,a,b}\lambda^{\frac{n-p-pa}{p}}\Big(1+|\lambda x|^{\frac{p\gamma(n-p-pa)}{(p-1)(n-p\gamma)}}\Big)^{1-\frac{n}{p\gamma}}\Big|\lambda>0,\gamma=1+a-b\right\},\quad(a,b)\neq (0,0),
    \end{equation*}
    \begin{equation*}
        \mathcal{M}^s_{p,0,0}:=\left\{A_{p,0,0}\lambda^{\frac{n-p}{p}}\Big(1+|\lambda (x-x_0)|^{\frac{p}{p-1}}\Big)^{1-\frac{n}{p}}\Big|\lambda>0,x_0\in\R^n\right\}.
    \end{equation*}
The number $A_{p,a,b}$ is a positive normalization factor such that, for any $v\in\mathcal{M}^s_{p,a,b}$, we have
\begin{equation*}
    \int_{\R^n}|x|^{-pa}|\nabla v|^pdx=\int_{\R^n}|x|^{-qb}|v|^qdx=S(p,a,b)^{\frac{pq}{q-p}}.
\end{equation*}
Functions in $\mathcal{M}^s_{p,a,b}$ are usually referred to as Talenti bubbles. It is straightforward to observe that when $(a,b)\neq (0,0)$, $\mathcal{M}_{p,a,b}^s$ is a one-dimensional scaling-invariant manifold. When $(a,b)=(0,0)$, $\mathcal{M}_{p,a,b}^s$ is an $(n+1)$-dimensional scaling-translation-invariant manifold. We remark that in some special parameter regions, $\mathcal{M}_{p,a,b}^s$ gives all positive solutions of \eqref{el}. We refer to \cite{Aub,Cir,Dol1,Le,Sci,Tal,Vet} and the references therein for further discussions.

\vskip0.1in
This stability problem typically refers to whether being almost a positive solution of \eqref{el} implies being close to a suitable sum of Talenti bubbles. This problem can be traced back to the well-known global compactness principle of Struwe \cite{Str}, which provided a qualitative estimate in the case $p=2,a=b=0$. Thanks to the global compactness principle, one can replace the assumption that $u$ is positive by that $u$ is initially close to a suitable sum of Talenti bubble. The first quantitative stability was given by Ciraolo, Figalli and Maggi \cite{Cir1}. They assumed $u$ is close to a single Talenti bubble and obtained the linear estimate
\begin{equation*}
	\inf_{v\in\mathcal{M}_{2,0,0}^s}\norm*{u-v}_{D_0^2}\leq C\left\Vert \Delta u +|u|^{\frac{4}{n-2}}u\right\Vert_{D_0^{-2}}.
\end{equation*}
Here, the space $D_a^{-p}(\R^n)$ denotes the dual space of $D_a^p(\R^n)$. Subsequently, Figalli and Glaudo \cite{Fig0} established analogous results in the case when $u$ is close to a sum of weak-interacting Talenti bubbles and $3\leq n\leq 5$. They discovered that these stabilities are strongly dependent on the dimension $n$ and showed that linear estimate does not hold when $n\geq 6$. The remaining case $n\geq 6$ was eventually solved by Deng, Sun and Wei \cite{Den0} using the method of finite dimensional reduction. They established logarithmic and sublinear estimates for the case $n=6$ and $n>6$, respectively. Recently, Wei and Wu \cite{Wei}
extended the results in \cite{Den0} to the general equation \eqref{el} in the case $p=2$. We refer to \cite{Ary,De,Liu,Pic,Zha,Zho} for various interesting results on other Sobolev-type inequalities.

\vskip0.1in
To the best of our knowledge, the stability problem in the case $p\neq 2$ has not been investigated so far. Our third main result aims to fill this gap in a weak sense. Let us first give a definition:
\begin{definition}\label{def1}
    Assume $a,b,p,q$ satisfy the relations in \eqref{ckn1}. For any $u\in D_a^p(\R^n)$, we define
    \begin{equation}\label{def2}
        P_u:=\left\{V\in\mathcal{M}_{p,a,b}^s\;|\;\int_{\R^n}|x|^{-qb}|V|^{q-2}Vu=\max_{W\in\mathcal{M}_{p,a,b}^s}\int_{\R^n}|x|^{-qb}|W|^{q-2}Wu\right\}.
    \end{equation}
\end{definition}
Note that, when $p=2$, the set $P_u$ consists precisely of functions $V$ such that
\begin{equation*}
        \inf_{W\in\mathcal{M}_{2,a,b}^s}\norm*{u-W}_{D_a^2}=\norm*{u-V}_{D_a^2}.
\end{equation*}
Heuristically, $P_u$ can be used to describe the distance between $u$ and $\mathcal{M}_{p,a,b}^s$. The main reason for using $P_u$ instead of directly minimizing the distance is that $P_u$ provides a suitable framework for the applications of spectral gap inequalities. Here is our main theorem.

\begin{theorem}\label{thm5}
    Assume $a,b,p,q$ satisfy the relations in \eqref{ckn1} and $p>2$. Then there exist three positive constants $c_1<C_1$ and $\delta$, depending only on $a,b,p,q,n$, such that, for any function $u\in D_a^p(\R^n)$ and $V\in P_u$, if
    \begin{equation}\label{sta5}
        \inf_{v\in\mathcal{M}_{p,a,b}^s}\norm*{u-v}_{D_a^p}\leq \delta,
    \end{equation}
    then the following two estimates hold:
  \begin{align}
        \label{sta6}
        c_1\int_{\R^n}|x|^{-pa}|\nabla V|^{p-2}|\nabla\rho|^2\,dx-C_1\norm*{\rho}^{p}_{D_a^p}
        \leq\norm*{H(u)}_{D_a^{-p}}\norm*{\rho}_{D_a^p},
    \end{align}
   \begin{align}\label{sta7}
        c_1\norm*{\rho}^{p}_{D_a^p}-C_1\int_{\R^n}|x|^{-pa}|\nabla V|^{p-2}|\nabla\rho|^2\,dx
        \leq\norm*{H(u)}_{D_a^{-p}}\norm*{\rho}_{D_a^p}.
    \end{align}
    Here $\rho=u-\mu V$ and $\mu=\frac{\int_{\R^n}|x|^{-qb}|V|^{q-2}Vu\,dx}{\int_{\R^n}|x|^{-qb}|V|^{q}\,dx}$.
\end{theorem}
A direct consequence is the following alternative-type stability result.
\begin{corollary}\label{cor}
    Assume $a,b,p,q$ satisfy the relations in \eqref{ckn1} and $p>2$. Then there exist three positive constants $\kappa, \eta$ and $\delta$, depending only on $a,b,p,q,n$, such that, for any function $u\in D_a^p(\R^n)$ and $V\in P_u$, if
    \begin{equation*}
        \inf_{v\in\mathcal{M}_{p,a,b}^s}\norm*{u-v}_{D_a^p}\leq \delta,
    \end{equation*}
    then at least one of the following two estimates holds:
    \begin{equation}\label{cor1}
        \kappa\norm*{u-V}_{D_a^p}^{p-1}\leq\norm*{H(u)}_{D_a^{-p}},
    \end{equation}
    \begin{equation}\label{cor2}
        \kappa\norm*{u_t-V}_{D_a^p}^{p-1}\leq\norm*{H(u_t)}_{D_a^{-p}}\quad\text{for any }0<t\leq\eta.
    \end{equation}
    Here $u_t:=tu+(1-t)V$.
\end{corollary}
\begin{remark}
    It is worth mentioning that Theorem \ref{thm5} makes sense only if the value of
    \begin{equation}\label{cor3}
        \frac{\norm*{\rho}^{p}_{D_a^p}}{\int_{\R^n}|x|^{-pa}|\nabla V|^{p-2}|\nabla\rho|^2\,dx}=:A_u
    \end{equation}
    falls outside the interval $\left[\frac{c_1}{C_1},\frac{C_1}{c_1}\right]$. In fact, if $A_u\notin \left[\frac{c_1}{2C_1},\frac{2C_1}{c_1}\right]$, we deduce the uniform estimate \eqref{cor1}. If $A_u\in \left[\frac{c_1}{2C_1},\frac{2C_1}{c_1}\right]$, we can obtain \eqref{cor2}.
\end{remark}
\begin{remark}
    Our derivation of Theorem \ref{thm5} and Corollary \ref{cor} was partly motivated by the work of Figalli and Neumayer \cite{Fig}. In their paper, they investigated the stability of the Sobolev inequality for the case $p\neq2$. They first derived an alternative-type result and then, through careful arguments, demonstrated that quantitative stability can be achieved regardless of the specific situation. In contrast to their findings, we wonder whether \eqref{cor1} always holds. The connection between \eqref{cor2} and \eqref{cor1} is quite implicit in our case.
\end{remark}
\begin{remark}
    Our proofs of Theorem \ref{thm5} and Corollary \ref{cor} do not actually utilize the transformation method; instead, they rely on careful estimates and some concepts from \cite{Fig0,Fig}. We include these results not only to provide a more comprehensive analysis of the stability of the Caffarelli-Kohn-Nirenberg inequality but also with the hope that our work may inspire future research on related topics.
\end{remark}
\vskip0.2in
After completing the first version of this paper \cite{Zho0} and posting it on arXiv, we learned that Deng and Tian \cite{Den3} obtained the optimal gradient stability of \eqref{ckn} when $(a,b,p,q)$ satisfy the relations in \eqref{ckn1} and $a=b>0$, thereby answering one of our questions. Specifically, they established that \eqref{sta2} holds with the exponent $\max\{2,p\}$, which is indeed optimal. Building on this optimal result, several of our findings related to the case $p\neq 2,a=b>0$ can be reinforced.

More recently, in \cite{Liu0}, Liu and Zhang established quantitative stability for critical points of the $p$-Sobolev inequality by modifying the techniques from \cite{Fig1}, thereby answering one of our questions regarding the nonweighted version. Specifically, they proved that the estimate \eqref{cor1} always holds in the nonweighted case $a=b=0$. Moreover, they also addressed the case $1<p<2$, which our method here fails to handle. It seems that their method can be easily adapted to the general case. 

\subsection{About various Hardy-type inequalities}\label{sec1.3}
Recall the classical Hardy inequality: for $1<p<n$ and $u\in D_0^p(\R^n)$,
\begin{align}\label{kk1}
    \int_{\R^n}|\nabla u|^pdx\geq\left(\frac{n-p}{p}\right)^p\int_{\R^n}\frac{|u|^p}{|x|^p}dx.
\end{align}
The constant $\left(\frac{n-p}{p}\right)^p$ is optimal; however, unlike the inequality \eqref{ckn}, it cannot be attained by any nonzero function. The Hardy inequality \eqref{kk1} is one of the most frequently used inequalities in analysis and has been studied intensively and extensively in the literature. We refer readers to the monographs \cite{Bal0,Gho00,Kuf,Kuf1} and the references therein for related topics.

\vskip0.1in
Regarding the refinement of \eqref{kk1}, since it does not possess nontrivial extremizers, one may hope that the difference between its two sides can directly control suitable norms of the function $u$, rather than merely measuring the distance between $u$ and a certain function. In the celebrated paper \cite{Bre1}, Br\'ezis and V\'azquez first obtained an improved Hardy inequality on bounded domains. Specifically, they proved that for any bounded domain $\Omega\subset\R^n$, with $n\geq 3$, $1<q<\frac{2n}{n-2}$, and every $u\in D_0^2(\Omega)$,
\begin{align}\label{kk2}
    \int_{\Omega}|\nabla u|^2dx-\left(\frac{n-2}{2}\right)^2\int_{\Omega}\frac{|u|^2}{|x|^2}dx\geq C(n,q)|\Omega|^{1-\frac{2}{n}-\frac{2}{q}}\left(\int_{\Omega}|u|^qdx\right)^{\frac{2}{q}}.
\end{align}
Moreover, in the case $q=2$ and $\Omega$ being a ball, they identified the optimal constant $C(n,2)$ and showed that it cannot be attained by any nonzero function. In \cite{Bre1}, they raised the question of whether the term on the right-hand side of \eqref{kk2} can be further improved, and whether the difference between the left-hand side and right-hand side of \eqref{kk2} can also control suitable norms of the function $u$. Their questions have since garnered significant attention and motivated many researchers to work on refinements of Hardy and Hardy-type inequalities. We refer interested readers to \cite{Abd,Adi0,Adi1,Barb0,Barb0.5,Barb,Bre1,Fil,Gaz0.5,Gki,Psa,Vaz,Wan,Zog}, to name just a few.

\vskip0.1in

Motivated by the work in \cite{Bre1}, a primary objective of this paper is to investigate various improved Hardy-type inequalities. Utilizing a class of simple yet effective transformations, we establish several new sharp Hardy-type inequalities and derive their refinements along with corresponding explicit stability constants. A detailed introduction to our main results is provided in Section \ref{sec3}, and here we offer a brief summary of our findings.

\vskip0.1in
In Subsection \ref{sec3.1}, we survey some classical improved Hardy inequalities from \cite{Adi1,Bre1,Fil,Gaz0.5,Gki,Stu,Zog}. By utilizing suitable transformations, we extend these inequalities to their weighted versions and derive various improved Hardy-Sobolev inequalities. Our method also provides explicit stability constants.

\vskip0.1in
In Subsection \ref{sec3.2}, we examine the $L^p$-logarithmic Sobolev inequality established in \cite{del0,Fuj,Gen,Gro,Weis}. We develop its weighted versions and identify the corresponding sharp constants and extremizers. Additionally, we investigate potential stability results, at least for centrally symmetric functions. Inspired by recent work in \cite{Bal}, we also consider the $L^p$-logarithmic Sobolev inequality involving a log-concave homogeneous weight on convex cones.
\vskip0.1in
In Subsection \ref{sec3.3}, we address the logarithmic Hardy inequality established in \cite{del1}. There are relatively few results concerning this inequality, and we aim to derive results regarding the symmetry and symmetry breaking of its extremizers. Our results complement those previously presented in \cite{Dol0,Dol0.5,Dol2}.

\vskip0.1in
In Subsection \ref{sec3.4}, we focus on the Hardy-Morrey inequality established in \cite{Psa}. This inequality refines the Hardy inequality \eqref{kk1} for $p>n$ by incorporating a sharp H\"older norm of $u$. Here, we obtain a weighted version: the Hardy-Sobolev-Morrey inequality with a H\"older refinement. Moreover, motivated by \cite{Gki}, we derive a series expansion of the Hardy-Sobolev-Morrey inequality.

\vskip0.1in
In Subsection \ref{sec3.5}, we extend some intriguing Hardy-Sobolev interpolation inequalities recently established in \cite{Die} to their weighted versions. These inequalities provide an interesting connection between the optimal Hardy inequality \eqref{kk1} and the critical Sobolev inequality.

\vskip0.1in
In Subsection \ref{sec3.6}, we recover some stability results recently established in \cite{Caz1}, which concern a class of interpolated Caffarelli-Kohn-Nirenberg inequalities. Our method is not only straightforward but also yields simple and explicit stability constants.

\vspace{20pt}
\addcontentsline{toc}{subsection}{Organization of this paper}\noindent$\textbf{Organization of this paper.}$ In Section \ref{sec2}, we focus on the stability of the Caffarelli-Kohn-Nirenberg inequality \eqref{ckn}. In Subsection \ref{sec2.1}, we first introduce a simple yet useful transformation inspired by Horiuchi \cite{Hor}, and then use it to prove Theorem \ref{thm1}, Theorem \ref{thm2}, Theorem \ref{thm3}, and Theorem \ref{thm4}. In Subsection \ref{sec2.2}, based on the stability results eatablished in Subsection \ref{sec2.1}, we derive improved Sobolev-type embeddings with weak Lebesgue norms (Corollary \ref{ccc1}, Corollary \ref{ccc6}, and Corollary \ref{ccc8}). In Subsection \ref{sec2.3}, we apply some careful expansion techniques to establish Theorem \ref{thm5}. Adapting an interpolation technique from \cite{Fig} allows us to derive Corollary \ref{cor}. Finally, in Section \ref{sec3}, we establish various improved Hardy-type inequalities. In each subsection, we introduce one type of Hardy-type inequality, state our main results, and provide the proofs.

\vspace{10pt}
\addcontentsline{toc}{subsection}{Notations}\noindent$\textbf{Notations.}$ In the following proofs, for simplicity, we will consistently use $C$ and $c$ to denote positive quantities depending only on the parameters (for example, $a,b,p,q$) and the dimension $n$. The values of $C$ and $c$ may vary from line to line. Additionally, when referring to a positive variable $A$, we will use the notation $o_A(1)$ to indicate any quantity that tends towards zero as $A$ approaches zero.

\vspace{10pt}
\addcontentsline{toc}{subsection}{Acknowledgement}\noindent$\textbf{Acknowledgement.}$ The authors wish to thank Prof R. Frank for his insightful comments on Corollary \ref{ccc1} in the original version of this paper. We also thank the anonymous referees for their careful reading of the manuscript and for their valuable suggestions, which have significantly improved the paper.
\vskip0.4in

\section{Stability of the Caffarelli-Kohn-Nirenberg inequality}\label{sec2}
This section is devoted to the stability of the Caffarelli-Kohn-Nirenberg inequality \eqref{ckn}. Subsection \ref{sec2.1} focuses on its stability in the functional setting, while Subsection \ref{sec2.3} addresses stability in the critical point setting. Based on the results established in Subsection \ref{sec2.1}, various improved Sobolev-type embeddings are derived in Subsection \ref{sec2.2}.

\subsection{Stability in the functional setting}\label{sec2.1}
 In this subsection, we focus on the functional stability of the Caffarelli-Kohn-Nirenberg inequality \eqref{ckn}. We aim to prove Theorem \ref{thm1}, Theorem \ref{thm2}, Theorem \ref{thm3}, and Theorem \ref{thm4}. Before commencing the proof, let us make some important observations. Assume that $a,b,p,q$ satisfy the relations in \eqref{ckn1} and that $a>0$. For any function $u\in D_a^p(\R^n)$, utilizing polar coordinates, we have:
 \begin{equation}\label{for1}
     \int_{\R^n}|x|^{-pa}|\nabla u|^pdx=\int_{\S^{n-1}}\int_0^{+\infty}\left(|\nabla_ru|^2+r^{-2}|\nabla_\theta u|^2\right)^{\frac{p}{2}}r^{n-1-ap}drd\theta
 \end{equation}
 and
 \begin{equation}\label{for2}
     \int_{\R^n}|x|^{-qb}|u|^qdx=\int_{\S^{n-1}}\int_0^{+\infty}|u|^q r^{n-1-bq}drd\theta.
 \end{equation}
It is straightforward to compute
\begin{equation}\label{for3}
    |\nabla_\theta u|^2=\sum\limits_{1\leq i<j\leq n}\left(x_i\partial_ju-x_j\partial_iu\right)^2.
\end{equation}
Let $k=\frac{n-p}{n-p-ap}$. Consider the new function $\Bar{u}(r\theta):=k^{\frac{1}{q}}u(r^k\theta)$ (this transformation traces back to the work in \cite{Hor}). We observe that the formulas \eqref{for1} and \eqref{for2} transform into:
\begin{equation}\label{for4}
    k^{1-p-\frac{p}{q}}\int_{\S^{n-1}}\int_0^{+\infty}\left(|\nabla_r\Bar{u}|^2+k^2r^{-2}|\nabla_\theta \Bar{u}|^2\right)^{\frac{p}{2}}r^{n-1}drd\theta
\end{equation}
and
\begin{equation}\label{for5}
    \int_{\S^{n-1}}\int_0^{+\infty}|\Bar{u}|^q r^{n-1-(b-a)q}drd\theta= \int_{\R^n}|x|^{-q(b-a)}|\Bar{u}|^qdx
\end{equation}
respectively. Since $a>0$, we have $k>1$. Thus, $\Bar{u}$ is a well-defined function in $D_0^p(\R^n)$. Moreover, it can be easily verified that the correspondence $u\leftrightarrow\Bar{u}$ provides a bijection between $\mathcal{M}_{p,a,b}$ and $\mathcal{M}_{p,0,b-a}$ (when $b=a$, we restrict to radially symmetric functions in $\mathcal{M}_{p,0,0}$ centered at the origin).

In the following, we will use the above transformation to provide the proofs of Theorem \ref{thm1}, Theorem \ref{thm2}, Theorem \ref{thm3}, and Theorem \ref{thm4}. Although some computations in these proofs are quite similar, we choose to present them in detail, not only for the convenience of the readers but also because our proofs serve as prototypes for obtaining sharp constants and extremizers, establishing stability results, and finding explicit relationships between stability constants. Later in Section \ref{sec3}, when encountering similar computations, we will omit certain details.

\begin{proof}[Proof of Theorem \ref{thm1}]
    By homogeneity, we can assume that
    \begin{equation*}
        \int_{\R^n}|x|^{-qb}|u|^qdx=1.
    \end{equation*}
    Using the above transformation, we have
    \begin{align}
            \int_{\R^n}|x|^{-pa}|\nabla u|^pdx=&\,k^{1-p-\frac{p}{q}}\int_{\S^{n-1}}\int_0^{+\infty}\left(|\nabla_r\Bar{u}|^2+k^2r^{-2}|\nabla_\theta \Bar{u}|^2\right)^{\frac{p}{2}}r^{n-1}drd\theta\nonumber\\
            \geq &\,k^{1-p-\frac{p}{q}}\int_{\S^{n-1}}\int_0^{+\infty}\left(|\nabla_r\Bar{u}|^2+r^{-2}|\nabla_\theta \Bar{u}|^2\right)^{\frac{p}{2}}r^{n-1}drd\theta\nonumber\\
            =&\,k^{1-p-\frac{p}{q}}\int_{\R^n}|\nabla \Bar{u}|^pdx\nonumber\\
            \geq&\,k^{1-p-\frac{p}{q}}S^p(p,0,b-a)\left(\int_{\R^n}|x|^{-q(b-a)}|\Bar{u}|^qdx\right)^{\frac{p}{q}}\nonumber\\
            =&\,k^{1-p-\frac{p}{q}}S^p(p,0,b-a)\nonumber\\
            =&\,S^p(p,a,b).\nonumber
    \end{align}
    The first inequality becomes equality precisely when $u$ is radially symmetric about the origin. The second inequality holds as an equality if and only if $\Bar{u}$ belongs to $\mathcal{M}_{p,0,b-a}$ (see, for example, \cite{Gho}). By combining these observations, we can easily conclude our results
\end{proof}

\begin{proof}[Proof of Theorem \ref{thm2}]
    By homogeneity, we can suppose that
    \begin{equation*}
        \int_{\R^n}|x|^{-qb}|u|^qdx=1.
    \end{equation*}
    Using similar arguments as in the previous proof and the stability estimate \eqref{DT} (see \cite{Den1} for detalis), we obtain
        \begin{align}
            \left(\int_{\R^n}|x|^{-pa}|\nabla u|^pdx\right)^{\frac{1}{p}}-S(p,a,b)\geq&\,k^{\frac{1}{p}-1-\frac{1}{q}}\left(\int_{\R^n}|\nabla \Bar{u}|^pdx\right)^{\frac{1}{p}}-S(p,a,b)\nonumber\\
            =&\,k^{\frac{1}{p}-1-\frac{1}{q}}\left(\left(\int_{\R^n}|\nabla \Bar{u}|^pdx\right)^{\frac{1}{p}}-S(p,0,b-a)\right)\nonumber\\
            \geq&\,c\inf_{v\in\mathcal{M}_{p,0,b-a}}\norm*{\nabla \Bar{u}-\nabla v}_{L^p}^{\max\{2,p\}}.\nonumber
        \end{align}
    Since the correspondence $u\leftrightarrow\Bar{u}$ provides a bijection between $\mathcal{M}_{p,0,b-a}$ and $\mathcal{M}_{p,a,b}$, we have
    \begin{align}
            &\inf_{v\in\mathcal{M}_{p,0,b-a}}\norm*{\nabla \Bar{u}-\nabla v}_{L^p}^p\nonumber\\
            &=\,\inf_{v\in\mathcal{M}_{p,0,b-a}}\int_{\S^{n-1}}\int_0^{+\infty}\left(|\nabla_r(\Bar{u}-v)|^2+r^{-2}|\nabla_\theta (\Bar{u}-v)|^2\right)^{\frac{p}{2}}r^{n-1}drd\theta\nonumber\\
            &\geq\,ck^{1-p-\frac{p}{q}}\inf_{v\in\mathcal{M}_{p,0,b-a}}\int_{\S^{n-1}}\int_0^{+\infty}\left(|\nabla_r(\Bar{u}-v)|^2+k^2r^{-2}|\nabla_\theta (\Bar{u}-v)|^2\right)^{\frac{p}{2}}r^{n-1}drd\theta\nonumber\\
            &=\,c\inf_{v\in\mathcal{M}_{p,a,b}}\int_{\S^{n-1}}\int_0^{+\infty}\left(|\nabla_r(u-v)|^2+r^{-2}|\nabla_\theta (u-v)|^2\right)^{\frac{p}{2}}r^{n-1-ap}drd\theta\nonumber\\
            &=\,c\inf_{v\in\mathcal{M}_{p,a,b}}\norm*{|x|^{-a}(\nabla u-\nabla v)}_{L^p}^p.\nonumber
    \end{align}
    Thus, we conclude that
    \begin{align}
        \left(\int_{\R^n}|x|^{-pa}|\nabla u|^pdx\right)^{\frac{1}{p}}-S(p,a,b)\geq c\inf_{v\in\mathcal{M}_{p,a,b}}\norm*{|x|^{-a}(\nabla u-\nabla v)}_{L^p}^{\max\{2,p\}}.\nonumber
    \end{align}
    The proof is complete.
\end{proof}
To prove Theorem \ref{thm3}, we need the following two lemmas. The first lemma is an elementary estimate.
\begin{lemma}\label{lem1}
    Let $\alpha,\beta$ be two positive constants. Then there exists a number $\varepsilon=\varepsilon(\alpha,\beta)$ such that, for any positive numbers $A,B,D,E,F$ with $B\geq \alpha,F\leq \beta,|A-D|+|E-B|\leq \varepsilon F$, one has
    \begin{equation*}
        \frac{A}{B}-1\geq F\quad\Longrightarrow\quad\frac{D}{E}-1\geq\frac{F}{2}.
    \end{equation*}
\end{lemma}
\begin{proof}
    Since $A\geq B(1+F)$, we have
    \begin{equation*}
        \frac{D}{E}\geq \frac{(1+F)B-|A-D|}{B+|B-E|}.
    \end{equation*}
    Therefore,
    \begin{equation*}
        \frac{D}{E}-1\geq\frac{F}{2}\quad\Longleftrightarrow\quad\frac{1}{2}FB\geq (|A-D|+|B-E|)\left(1+\frac{F}{2}\right).
    \end{equation*}
    We just need to take $\varepsilon=\frac{\alpha}{\beta+2}$.
\end{proof}
The second lemma is an integral estimate, which enables us to deal with the translation-invariant property of $\mathcal{M}_{p,0,0}$.
\begin{lemma}\label{lem2}
    Given $k>1$. Then there exists a positive number $c$ such that, for any nontrivial function $U\in \mathcal{M}_{p,0,0}$ centered at the origin and vector $\rho\in\R^n$, it holds that
    \begin{equation*}
        \begin{aligned}
            &\frac{1}{\norm*{U}_{L^{\frac{np}{n-p}}}}\left(\int_{\S^{n-1}}\int_0^{+\infty}\left(|\nabla_rU(\cdot+\rho)|^2+k^2r^{-2}|\nabla_\theta U(\cdot+\rho)|^2\right)^{\frac{p}{2}}r^{n-1}drd\theta\right)^\frac{1}{p}-S(p,0,0)\\
            \geq&\,c\left(\frac{\norm*{\nabla U-\nabla U(\cdot+\rho)}_{L^p}}{\norm*{U}_{L^{\frac{np}{n-p}}}}\right)^2.
        \end{aligned}
    \end{equation*}
\end{lemma}
\begin{proof}
    By homogeneity and dilation invariance, we can assume that
    \begin{equation*}
        U(x)=L\left(1+|x|^{\frac{p}{p-1}}\right)^{\frac{p-n}{p}},
    \end{equation*}
    where $L$ is a suitable constant such that $\norm*{U}_{L^{\frac{np}{n-p}}}=1$. For any fixed $\varepsilon>0$, if $|\rho|\geq \varepsilon$, it follows that the difference on the left side is bounded away from $0$. Since the term $\norm*{\nabla U-\nabla U(\cdot+\rho)}_{L^p}$ is also bounded above, the existence of the number $c$ is straightforward. Therefore, we only need to focus on the case when $|\rho|$ is small.

    In this scenario, we claim that the values on both sides are comparable to $|\rho|^2$. To demonstrate this, we can first apply the mean value theorem to the right-hand side. For the left-hand side, we need to estimate the difference between
    \begin{equation*}
        \left(\int_{\S^{n-1}}\int_0^{+\infty}\left(|\nabla_rU(\cdot+\rho)|^2+k^2r^{-2}|\nabla_\theta U(\cdot+\rho)|^2\right)^{\frac{p}{2}}r^{n-1}drd\theta\right)^\frac{1}{p}
    \end{equation*}
    and
    \begin{equation*}
        \left(\int_{\S^{n-1}}\int_0^{+\infty}\left(|\nabla_rU(\cdot+\rho)|^2+r^{-2}|\nabla_\theta U(\cdot+\rho)|^2\right)^{\frac{p}{2}}r^{n-1}drd\theta\right)^\frac{1}{p}.
    \end{equation*}
    Note that the values of these two integrals fall within the interval $[S(p,0,0),kS(p,0,0)]$, we can apply the basic inequality:
    \begin{equation*}
        c(x^p-y^p)\leq x-y\leq C(x^p-y^p)\quad \text{for any }S(p,0,0)\leq y\leq x\leq kS(p,0,0).
    \end{equation*}
    Now, it suffices to estimate
    \begin{equation*}
        \left(|\nabla_rU(\cdot+\rho)|^2+k^2r^{-2}|\nabla_\theta U(\cdot+\rho)|^2\right)^{\frac{p}{2}}-\left(|\nabla_rU(\cdot+\rho)|^2+r^{-2}|\nabla_\theta U(\cdot+\rho)|^2\right)^{\frac{p}{2}}
    \end{equation*}
    pointwise. Without loss of generality, we can assume $\rho=(|\rho|,0,\cdots,0)$. From \eqref{for3}, we have
    \begin{align}
         |\nabla_\theta U(x+\rho)|^2=&\sum\limits_{j=2}^n(x_1\partial_jU-x_j\partial_1U)^2\nonumber\\
         =&\,C|\rho|^2\sum\limits_{j=2}^n|x_j|^2|x+\rho|^{\frac{4-2p}{p-1}}\left(1+|x+\rho|^{\frac{p}{p-1}}\right)^{-\frac{2n}{p}}.\nonumber
    \end{align}
    Now our can apply the mean value theorem again to deduce the claim.
\end{proof}

\begin{proof}[Proof of Theorem \ref{thm3}]
    By homogeneity, we can assume that
    \begin{equation*}
        \int_{\R^n}|x|^{-qa}|u|^qdx=1.
    \end{equation*}
    As in the proof of Theorem \ref{thm2}, we obtain
    \begin{equation*}
         \left(\int_{\R^n}|x|^{-pa}|\nabla u|^pdx\right)^{\frac{1}{p}}-S(p,a,a)\geq c\inf_{v\in\mathcal{M}_{p,0,0}}\norm*{\nabla \Bar{u}-\nabla v}_{L^p}^{\max\{2,p\}}.
    \end{equation*}
    Note that we cannot transform $\Bar{u}$ back to $u$ directly because $\mathcal{M}_{p,0,0}$ is translation-invariant while $\M_{p,a,a}$ is not. We define $\mathcal{M}^R_{p,0,0}$ to be the subset of $\mathcal{M}_{p,0,0}$ consisting of all radially symmetric functions centered at the origin.

    When the function $u$ is centrally symmetric about the origin, we can show that
    \begin{equation*}
        \inf_{v\in\mathcal{M}_{p,0,0}}\norm*{\nabla \Bar{u}-\nabla v}_{L^p}\geq c\inf_{v\in\mathcal{M}^R_{p,0,0}}\norm*{\nabla \Bar{u}-\nabla v}_{L^p}.
    \end{equation*}
    Specifically, let $v_0\in\mathcal{M}_{p,0,0}^R$ and $x_0\in\R^n$ such that
    \begin{equation*}
        \inf_{v\in\mathcal{M}_{p,0,0}}\norm*{\nabla \Bar{u}-\nabla v}_{L^p}=\norm*{\nabla \Bar{u}-\nabla v_0(\cdot+x_0)}_{L^p}.
    \end{equation*}
    Then we have:
    \begin{equation*}
        \begin{aligned}
            \norm*{\nabla \Bar{u}-\nabla v_0}_{L^p}\leq&\, \norm*{\nabla \Bar{u}-\nabla v_0(\cdot+x_0)}_{L^p}+\norm*{\nabla v_0(\cdot+x_0)-\nabla v_0}_{L^p}\\
            \leq&\,\norm*{\nabla \Bar{u}-\nabla v_0(\cdot+x_0)}_{L^p}+C\norm*{\nabla v_0(\cdot+x_0)-\nabla v_0(\cdot-x_0)}_{L^p}\\
            \leq&\,C\norm*{\nabla \Bar{u}-\nabla v_0(\cdot+x_0)}_{L^p}+
            C\norm*{\nabla \Bar{u}-\nabla v_0(\cdot-x_0)}_{L^p}\\
            \leq&\,C\norm*{\nabla \Bar{u}-\nabla v_0(\cdot+x_0)}_{L^p}.
        \end{aligned}
    \end{equation*}
    The second inequality above follows from the mean value theorem. Therefore, \eqref{sta2} holds with the optimal exponent $\max\{2,p\}$.

    For a general function $u$, we need to use some careful arguments. Fix a number $0<\varepsilon\ll 1$ which will be determined later. From the concentration compactness principle (see for example \cite{Hor,Lio}), there exists a positive number $\delta<1$ such that, when $\norm*{\nabla \Bar{u}}_{L^p}<(1+\delta)S(p,0,0)$, we have
    \begin{equation*}
        \inf_{v\in\mathcal{M}_{p,0,0}}\norm*{\nabla \Bar{u}-\nabla v}_{L^p}\leq \varepsilon.
    \end{equation*}
    Note that if $\norm*{\nabla \Bar{u}}_{L^p}\geq(1+\delta)S(p,0,0)$, from the transformation we see that
    \begin{align}
        \norm*{u}_{D_a^p}\geq (1+\delta)S(p,a,a),\nonumber
    \end{align}
    which directly leads to the conclusion in \eqref{sta2}. Therefore, to prove \eqref{sta2}, it suffices to consider the case $\norm*{\nabla \Bar{u}}_{L^p}<(1+\delta)S(p,0,0)$. Now, if
    \begin{equation*}
        \inf_{v\in\mathcal{M}_{p,0,0}}\norm*{\nabla \Bar{u}-\nabla v}_{L^p}\geq \varepsilon\inf_{v\in\mathcal{M}^R_{p,0,0}}\norm*{\nabla \Bar{u}-\nabla v}_{L^p}^2,
    \end{equation*}
    then \eqref{sta2} holds. If not, we can find a function $U\in\mathcal{M}^R_{p,0,0}$ and a vector $\rho\in\R^n$ such that
    \begin{equation*}
        \norm*{\nabla \Bar{u}-\nabla U(\cdot+\rho)}_{L^p}=\inf_{v\in\mathcal{M}_{p,0,0}}\norm*{\nabla \Bar{u}-\nabla v}_{L^p}\leq \varepsilon
    \end{equation*}
    and
    \begin{equation*}
        \norm*{\nabla \Bar{u}-\nabla U(\cdot+\rho)}_{L^p}\leq \varepsilon\norm*{\nabla \Bar{u}-\nabla U}_{L^p}^2.
    \end{equation*}
    Since $\varepsilon$ is small, we have
    \begin{equation*}
        \norm*{\nabla \Bar{u}-\nabla U(\cdot+\rho)}_{L^p}\leq 2\varepsilon\norm*{\nabla U(\cdot+\rho)-\nabla U}_{L^p}^2,
    \end{equation*}
    and
    \begin{equation*}
        \frac{1}{2}\norm*{\nabla \Bar{u}-\nabla U}_{L^p}\leq \norm*{\nabla U(\cdot+\rho)-\nabla U}_{L^p}.
    \end{equation*}
    Now let us apply Lemma \ref{lem1} and Lemma \ref{lem2} with
    \begin{align}
            &A=\left(\int_{\S^{n-1}}\int_0^{+\infty}\left(|\nabla_rU(\cdot+\rho)|^2+k^2r^{-2}|\nabla_\theta U(\cdot+\rho)|^2\right)^{\frac{p}{2}}r^{n-1}drd\theta\right)^\frac{1}{p},\nonumber\\
            &B=S(p,0,0)\norm*{U}_{L^{\frac{np}{n-p}}},\nonumber\\
            &D=\left(\int_{\S^{n-1}}\int_0^{+\infty}\left(|\nabla_r\Bar{u}|^2+k^2r^{-2}|\nabla_\theta \Bar{u}|^2\right)^{\frac{p}{2}}r^{n-1}drd\theta\right)^\frac{1}{p},\nonumber\\
            &E=S(p,0,0)\norm*{\Bar{u}}_{L^{\frac{np}{n-p}}}=S(p,0,0),\nonumber\\
            &F=c\left(\frac{\norm*{\nabla U-\nabla U(\cdot+\rho)}_{L^p}}{\norm*{U}_{L^{\frac{np}{n-p}}}}\right)^2,\nonumber
    \end{align}
    provided $\varepsilon$ is sufficiently small. Note that $|A-D|+|B-E|$ can be estimated directly by the Minkowski inequality. Therefore, we obtain
    \begin{align}
        \left(\int_{\R^n}|x|^{-pa}|\nabla u|^pdx\right)^{\frac{1}{p}}-S(p,a,a)=&\,c\left(\frac{D}{E}-1\right)
        \geq\,cF\nonumber\\
        \geq& \,c\norm*{\nabla U(\cdot+\rho)-\nabla U}_{L^p}^2
        \geq\, c\norm*{\nabla \Bar{u}-\nabla U}_{L^p}^2\nonumber\\
        \geq&\, c\inf_{v\in\mathcal{M}_{p,a,a}}\norm*{|x|^{-a}(\nabla u-\nabla v)}_{L^p}^2\nonumber\\
        \geq&\,c\inf_{v\in\mathcal{M}_{p,a,a}}\norm*{|x|^{-a}(\nabla u-\nabla v)}_{L^p}^{\max\{4,2p\}}.\nonumber
    \end{align}
    The proof is complete.
\end{proof}

\begin{proof}[Proof of Theorem \ref{thm4}]
    $(i)$ Without loss of generality, we assume
    \begin{equation*}
        \limsup_{l\rightarrow+\infty}K(p_l,a_l,b_l)=\lim_{l\rightarrow+\infty}K(p_l,a_l,b_l)=:K_0.
    \end{equation*}
    We only need to consider the case $K_0>0$. For any nonzero function $u\in C_c^{\infty}(\R^n)$, by definition, we have
    \begin{equation*}
        \frac{\norm*{|x|^{-a_l}\nabla u}_{L^{p_l}}}{\norm*{|x|^{-b_l}u}_{L^{q_l}}}-S(p_l,a_l,b_l)\geq K(p_l,a_l,b_l)\inf_{v\in\mathcal{M}_{p_l,a_l,b_l}}\left(\frac{\norm*{|x|^{-a_l}(\nabla u-\nabla v)}_{L^{p_l}}}{\norm*{|x|^{-a_l}\nabla u}_{L^{p_l}}}\right)^{\max\{2,{p_l}\}}
    \end{equation*}
    Let $v_l\in \mathcal{M}_{p_l,a_l,b_l}$ be such that
    \begin{equation*}
        \norm*{|x|^{-a_l}(\nabla u-\nabla v_l)}_{L^{p_l}}=\inf_{v\in\mathcal{M}_{p_l,a_l,b_l}}\norm*{|x|^{-a_l}(\nabla u-\nabla v)}_{L^{p_l}}.
    \end{equation*}
    Then we can deduce
    \begin{equation*}
        \frac{\norm*{|x|^{-a_l}\nabla u}_{L^{p_l}}}{\norm*{|x|^{-b_l}u}_{L^{q_l}}}-S(p_l,a_l,b_l)\geq K(p_l,a_l,b_l)\left(\frac{\norm*{|x|^{-a_l}(\nabla u-\nabla v_l)}_{L^{p_l}}}{\norm*{|x|^{-a_l}\nabla u}_{L^{p_l}}}\right)^{\max\{2,{p_l}\}}.
    \end{equation*}
    Since $v_l$ takes the form given in \eqref{ckn3} and $u$ is fixed, there exists a subsequence of $\{v_l\}$, still denoted by $\{v_l\}$, such that $v_l\rightarrow w$ and $\nabla v_l\rightarrow \nabla w$ almost everywhere in $\R^n$ for some $w$ belongs to $\mathcal{M}_{p,a,b}$. Now, letting $l\rightarrow+\infty$, by continuity and the Fatou's lemma, we obtain
    \begin{equation*}
    \begin{aligned}
        \frac{\norm*{|x|^{-a}\nabla u}_{L^p}}{\norm*{|x|^{-b}u}_{L^q}}-S(p,a,b)\geq&\,K_0\left(\frac{\norm*{|x|^{-a}(\nabla u-\nabla w)}_{L^p}}{\norm*{|x|^{-a}\nabla u}_{L^p}}\right)^{\max\{2,p\}}\\
        \geq&\, K_0\inf_{v\in\mathcal{M}_{p,a,b}}\left(\frac{\norm*{|x|^{-a}(\nabla u-\nabla v)}_{L^p}}{\norm*{|x|^{-a}\nabla u}_{L^p}}\right)^{\max\{2,p\}}.
    \end{aligned}
    \end{equation*}
    From a simple density argument, we see that the above estimate holds for any $u\in D_a^p(\R^n)$. Therefore, we conclude that $K(p,a,b)\geq K_0$.

    $(ii)$ Set $h=\frac{n-p-a_1p}{n-p-a_2p}>1$. For any $u\in D_{a_2}^p(\R^n)$, we define $\hat{u}(r\theta):=h^{\frac{1}{q}}u(r^h\theta)$. Following similar arguments as in equations \eqref{for1}-\eqref{for5}, we have:
    \begin{align}
            \int_{\R^n}|x|^{-pa_2}|\nabla u|^pdx=&\int_{\S^{n-1}}\int_0^{+\infty}\left(|\nabla_ru|^2+r^{-2}|\nabla_\theta u|^2\right)^{\frac{p}{2}}r^{n-1-a_2p}drd\theta\nonumber\\
            =&\,h^{1-p-\frac{p}{q}}\int_{\S^{n-1}}\int_0^{+\infty}\left(|\nabla_r\hat{u}|^2+h^2r^{-2}|\nabla_\theta \hat{u}|^2\right)^{\frac{p}{2}}r^{n-1-a_1p}drd\theta\nonumber\\
            \geq&\,h^{1-p-\frac{p}{q}}\int_{\S^{n-1}}\int_0^{+\infty}\left(|\nabla_r\hat{u}|^2+r^{-2}|\nabla_\theta \hat{u}|^2\right)^{\frac{p}{2}}r^{n-1-a_1p}drd\theta\nonumber\\
            =&\,h^{1-p-\frac{p}{q}}\int_{\R^n}|x|^{-pa_1}|\nabla \hat{u}|^pdx,\nonumber
    \end{align}
    and
    \begin{align}
        \int_{\R^n}|x|^{-qb_2}|u|^qdx=&\int_{\S^{n-1}}\int_0^{+\infty}|u|^q r^{n-1-b_2q}drd\theta\nonumber\\
        =&\int_{\S^{n-1}}\int_0^{+\infty}|\hat{u}|^q r^{n-1-b_1q}drd\theta\nonumber\\
        =&\int_{\R^n}|x|^{-qb_1}|\hat{u}|^qdx.\nonumber
    \end{align}
    Therefore, for any $u\in D_{a_2}(\R^n)$ with
    \begin{equation*}
        \int_{\R^n}|x|^{-qb_2}|u|^qdx=1,
    \end{equation*}
    we can compute
    \begin{align}\label{qqq}
            &\left(\int_{\R^n}|x|^{-pa_2}|\nabla u|^pdx\right)^{\frac{1}{p}}-S(p,a_2,b_2)\nonumber\\\geq&\,h^{\frac{1}{p}-1-\frac{1}{q}}\left(\left(\int_{\R^n}|x|^{-pa_1}|\nabla \hat{u}|^pdx\right)^{\frac{1}{p}}-S(p,a_1,b_1)\right)\nonumber\\
            \geq&\,h^{\frac{1}{p}-1-\frac{1}{q}}K(p,a_1,b_1)\inf_{v\in\mathcal{M}_{p,a_1,b_1}}\norm*{ \hat{u}-v}_{D_{a_1}^p}^{\max\{2,p\}}\nonumber\\
            \geq&\,h^{\frac{1}{p}-1-\frac{1}{q}}K(p,a_1,b_1)\cdot h^{\left(\frac{1}{q}-\frac{1}{p}\right)\max\{2,p\}}\inf_{v\in\mathcal{M}_{p,a_2,b_2}}\norm*{ u-v}_{D_{a_2}^p}^{\max\{2,p\}}\nonumber\\
            =&\,h^{-1-\max\{1,p-1\}\frac{\gamma}{n}}K(p,a_1,b_1)\inf_{v\in\mathcal{M}_{p,a_2,b_2}}\norm*{ u-v}_{D_{a_2}^p}^{\max\{2,p\}}.
        \end{align}
    Thus, we obtain
    \begin{equation*}
        K(p,a_2,b_2)\geq h^{-1-\max\{1,p-1\}\frac{\gamma}{n}}K(p,a_1,b_1),
    \end{equation*}
    which is equivalent to \eqref{sta4}. Moreover, it is easy to see that the first and the third inequalities in \eqref{qqq} cannot be equal simultaneously. Therefore, if $K(p,a_2,b_2)$ is achieved by some function, the estimate \eqref{sta4} is strict.
\end{proof}

\subsection{Improved Sobolev-type embeddings}\label{sec2.2}
In this subsection, we focus on Corollary \ref{ccc1}, Corollary \ref{ccc6}, and Corollary \ref{ccc8}. The basic idea is to combine the stability results established in the previous subsection with some integral estimates.
\begin{proof}[Proof of Corollary \ref{ccc1}]
    Due to the scaling-invariance of \eqref{ccc2}, it suffices to consider the case $|\Omega|=1$. By homogeneity, we assume that $\norm*{|x|^{-a}\nabla u}_{L^p(\Omega)}=1$. Let $\rho$ be a small positive constant that depends only on $a,b,p,q,n$ and will be determined later. Since $|\Omega|=1$, by the Hardy-Sobolev and H\"older inequalities, the norm $\norm*{|x|^{-a}u}_{L_w^{p_1}(\Omega)}$ is bounded above. If
    \begin{equation*}
        \inf_{v\in\mathcal{M}_{p,a,b}}\norm*{u-v}_{D_a^p(\R^n)}\geq \rho,
    \end{equation*}
    then \eqref{ccc2} follows directly from the gradient stability. If not, suppose
    \begin{equation*}
        \inf_{v\in\mathcal{M}_{p,a,b}}\norm*{u-v}_{D_a^p(\R^n)}=\norm*{u-AV[\lambda,x_0]}_{D_a^p(\R^n)}\leq\rho,
    \end{equation*}
    where $A\in\R,\lambda>0,x_0\in\R^n$ (if $(a,b)\neq (0,0)$, then $x_0\equiv 0$) and $V[\lambda,x_0](x):=B\lambda^{\frac{n-p-pa}{p}}\Big(1+|\lambda (x-x_0)|^{\frac{p\gamma(n-p-pa)}{(p-1)(n-p\gamma)}}\Big)^{1-\frac{n}{p\gamma}}$. Here, $B$ is a fixed normalization factor such that $\norm*{V[\lambda,x_0]}_{D_a^p(\R^n)}=1$. From the Minkowski inequality, we have $|A-1|\leq \rho$. Therefore, we obtain
    \begin{equation*}
        \rho\geq \norm*{u-AV[\lambda,x_0]}_{D_a^p(\R^n)}\geq \norm*{AV[\lambda,x_0]}_{D_a^p(\R^n\backslash\Omega)}\geq (1-\rho)\norm*{V[\lambda,x_0]}_{D_a^p(\R^n\backslash\Omega)}.
    \end{equation*}
    Let us make two simple observations without proofs:

    $(i)$ There exists a nonincreasing function $\Lambda:(0,+\infty)\rightarrow [0,+\infty)$ such that $\Lambda(t)\rightarrow +\infty$ as $t\rightarrow 0$ and for any domain $\Omega$ with $|\Omega|=1$, if
    \begin{equation*}
        \left(\int_{\R^n\backslash\Omega}|x|^{-pa}|\nabla V[\lambda,x_0]|^pdx\right)^{\frac{1}{p}}\leq \rho_0
    \end{equation*}
    holds for some $\rho_0>0,\lambda>0,x_0\in\R^n$ (if $a\neq 0$, let $x_0=0$), then $\lambda\geq \Lambda(\rho_0).$

    $(ii)$ There exist two positive constants $\lambda_0$ and $c_0$ such that for any $\lambda>\lambda_0,x_0\in\R^n$ (if $a\neq 0$, let $x_0=0$) and any domain $\Omega$ with $|\Omega|=1$, it holds that
    \begin{equation*}
        \left(\int_{\R^n\backslash\Omega}|x|^{-pa}|\nabla V[\lambda,x_0]|^pdx\right)^{\frac{1}{p}}\geq c_0\lambda^{-\frac{n-p-pa}{p(p-1)}}.
    \end{equation*}
    Note that a direct computation yields
    \begin{equation*}
        \norm*{|x|^{-a}V[\lambda,x_0]}_{L_w^{p_1}(\R^n)}\leq C\lambda^{-\frac{n-p-pa}{p(p-1)}}.
    \end{equation*}
    Thus, if we set $\rho$ sufficiently small, it holds that
    \begin{equation*}
        \norm*{|x|^{-a}V[\lambda,x_0]}_{L_w^{p_1}(\R^n)}\leq C\norm*{u-AV[\lambda,x_0]}_{D_a^p(\R^n)}
    \end{equation*}
    Now we can easily deduce
    \begin{equation*}
        \begin{aligned}
            \norm*{|x|^{-a}u}_{L_w^{p_1}(\Omega)}\leq &\,\norm*{|x|^{-a}AV[\lambda,x_0]}_{L_w^{p_1}(\R^n)}+\norm*{|x|^{-a}(u-AV[\lambda,x_0])}_{L_w^{p_1}(\Omega)}\\
            \leq&\,C\norm*{u-AV[\lambda,x_0]}_{D_a^p(\R^n)}+C\norm*{|x|^{-a}(u-AV[\lambda,x_0])}_{L^{\frac{np}{n-p}}(\Omega)}\\
            \leq &\,C\norm*{u-AV[\lambda,x_0]}_{D_a^{p}(\R^n)}.
        \end{aligned}
    \end{equation*}
    Therefore, \eqref{ccc2} follows from the gradient stability. The result \eqref{ccc3} can be obtained in a similar manner.
\end{proof}
The proofs of Corollaries \ref{ccc6} and \ref{ccc8} rely on the following lemma about the projection of $D_a^p(\R^n)$ to $D_a^p(\Omega)$.
\begin{lemma}\label{lll}
    Assume $a,b,p,q$ satisfy the relations in \eqref{ckn1}. Let $\Omega\subset\R^n$ and $P:D_a^p(\R^n)\rightarrow D_a^p(\Omega)$ be the projection operator. Then for any positive function $V\in\mathcal{M}_{p,a,b}$, one has $0\leq PV\leq V$ in $\R^n$.
\end{lemma}
\begin{proof}
    Note that $V-PV$ satisfies the equation:
    \begin{equation*}
        \begin{cases}
            -\div\left(|x|^{-pa}|\nabla (V-PV)|^{p-2}(V-PV)\right)=0\quad\text{in }\Omega,\\
            V-PV=V\quad\text{on }\partial\Omega.
        \end{cases}
    \end{equation*}
    From the weak maximum principle, we conclude that $V-PV\geq 0$. Since $V$ satisfies
    \begin{equation*}
        -\div\left(|x|^{-pa}|\nabla V|^{p-2}V\right)\geq 0\quad\text{in }\R^n.
    \end{equation*}
    By applying the comparison principle, we find that $V\geq V-PV$, which implies that $PV\geq0$ in $\R^n$.
\end{proof}

\begin{proof}[Proof of Corollary \ref{ccc6}]
    By homogeneity, we assume $\norm*{\nabla u}_{L^p(\Omega)}=1$. Let $\rho$ be a small positive constant that only depends on $p,n,R,|V\cap \S^{n-1}|$ and that will be determined later. By the H\"older and Sobolev inequalities, and the assumption $\sqrt{n}\leq p<n$, we have
    \begin{equation*}
        \norm*{u}_{L_w^{p_1}(\Omega)}\leq \norm*{u}_{L^{p_1}(\Omega)}\leq \norm*{u}_{L^p(\Omega)}^\beta \norm*{u}_{L^q(\Omega)}^{1-\beta}\leq C\lambda(p,\Omega)^{-\frac{\beta}{p}},
    \end{equation*}
    where $\beta$ satisfies the condition in \eqref{qwe}. If
    \begin{equation*}
        \inf_{v\in\mathcal{M}_{p,0,0}}\norm*{u-v}_{D_0^p(\R^n)}\geq \rho,
    \end{equation*}
    then \eqref{ccc7} is a direct consequence of gradient stability. If not, suppose
    \begin{equation*}
        \inf_{v\in\mathcal{M}_{p,0,0}}\norm*{u-v}_{D_0^p(\R^n)}=\norm*{u-AV[\lambda,x_0]}_{D_0^p(\R^n)}\leq \rho,
    \end{equation*}
    where $A\in\R,\lambda>0,x_0\in\R^n$ and $V[\lambda,x_0](x):=B\lambda^{\frac{n-p}{p}}\Big(1+|\lambda (x-x_0)|^{\frac{p}{p-1}}\Big)^{1-\frac{n}{p}}$. Here, $B$ is a fixed normalization factor such that $\norm*{V[\lambda,x_0]}_{D_0^p(\R^n)}=1$. It follows from the Minkowski inequality that $|A-1|\leq \rho$. Therefore, we deduce
    \begin{equation*}
        \rho\geq \norm*{u-AV[\lambda,x_0]}_{D_0^p(\R^n)}\geq \norm*{AV[\lambda,x_0]}_{D_0^p(\R^n\backslash\Omega)}\geq (1-\rho)\norm*{V[\lambda,x_0]}_{D_0^p(\R^n\backslash\Omega)}.
    \end{equation*}
    As in the proof of Corollary \ref{ccc1}, we make the following two observations: 
    
    For any fixed $R$ and $|V\cap \S^{n-1}|$, there exist a nonincreasing function $\Lambda:(0,+\infty)\rightarrow [0,+\infty)$ satisfying $\Lambda(t)\rightarrow +\infty$ as $t\rightarrow 0$ and two positive constants $\lambda_0$ and $c_0$ such that for any domain $\Omega$ satisfying the $(A_1)$ condition,

    $(i)$ if
    \begin{equation*}
        \left(\int_{\R^n\backslash\Omega}|\nabla V[\lambda,x_0]|^pdx\right)^{\frac{1}{p}}\leq \rho_0
    \end{equation*}
    holds for some $\rho_0>0,\lambda>0,x_0\in\R^n$, then $\lambda\geq \Lambda(\rho_0)$.

    $(ii)$ if $\lambda>\lambda_0$ and $x_0\in\R^n$, it holds
    \begin{equation*}
        \left(\int_{\R^n\backslash\Omega}|\nabla V[\lambda,x_0]|^pdx\right)^{\frac{1}{p}}\geq c_0\lambda^{-\frac{n-p}{p(p-1)}}.
    \end{equation*}
    Note that a direct computation yields
    \begin{equation*}
        \norm*{V[\lambda,x_0]}_{L_w^{p_1}(\R^n)}\leq C\lambda^{-\frac{n-p}{p(p-1)}}.
    \end{equation*}
    Thus, if we set $\rho$ sufficiently small, it holds
    \begin{equation*}
        \norm*{V[\lambda,x_0]}_{L_w^{p_1}(\R^n)}\leq C\norm*{u-AV[\lambda,x_0]}_{D_0^p(\R^n)}
    \end{equation*}
    Therefore, by utilizing Lemma \ref{lll}, we obtain
    \begin{equation*}
        \begin{aligned}
            \norm*{u}_{L_w^{p_1}(\Omega)}\leq&\;C\norm*{u-PAV[\lambda,x_0]
            ]}_{L_w^{p_1}(\Omega)}+C\norm*{PV[\lambda,x_0]}_{L_w^{p_1}(\Omega)}\\
            \leq&\,C\lambda(p,\Omega)^{-\frac{\beta}{p}}\norm*{u-PAV[\lambda,x_0]}_{D_0^p(\R^n)}+C\norm*{V[\lambda,x_0]}_{L_w^{p_1}(\Omega)}\\
            \leq&\,C\lambda(p,\Omega)^{-\frac{\beta}{p}}\left(\norm*{u-AV[\lambda,x_0]}_{D_0^p(\R^n)}+\norm*{AV[\lambda,x_0]-PAV[\lambda,x_0]}_{D_0^p(\R^n)}\right)\\
            &+C\norm*{u-AV[\lambda,x_0]}_{D_0^p(\R^n)}\\
            \leq&\,C\left(1+\lambda(p,\Omega)^{-\frac{\beta}{p}}\right)\norm*{u-AV[\lambda,x_0]}_{D_0^p(\R^n)}.
        \end{aligned}
    \end{equation*}
    Thus, \eqref{ccc7} now follows directly from the gradient stability.
\end{proof}

\begin{proof}[Proof of Corollary \ref{ccc8}]
    By homogeneity, we assume $\norm*{u}_{D_a^p(\Omega)}=1$. Let $\rho$ be a small positive constant that depends only on $a,b,p,q,n,R,|V\cap \S^{n-1}|$ and will be determined later. Using the H\"older and Hardy-Sobolev inequalities, along with the assumptions on $p_1$, we deduce
    \begin{equation*}
        \begin{aligned}
           \norm*{|x|^{-a}u}_{L_w^{p_1}(\Omega)}\leq&\, \norm*{|x|^{-a}u}_{L^{p_1}(\Omega)}\\
           \leq&\,\norm*{|x|^{-a}u}_{L^p(\Omega)}^\beta\norm*{|x|^{-a}u}_{L^\frac{np}{n-p}(\Omega)}^{1-\beta}\\
            \leq&\,C\lambda(p,\Omega)^{-\frac{\beta}{p}}\norm*{\nabla(|x|^{-a}u)}_{L^p(\Omega)}^\beta\norm*{|x|^{-a}\nabla u}_{L^p(\Omega)}^{1-\beta}\\
            \leq&\,C\lambda(p,\Omega)^{-\frac{\beta}{p}}\left(\norm*{|x|^{-a-1}u}_{L^p(\Omega)}+\norm*{|x|^{-a}\nabla u}_{L^p(\Omega)}\right)^\beta\norm*{|x|^{-a}\nabla u}_{L^p(\Omega)}^{1-\beta}\\
            \leq&\,C\lambda(p,\Omega)^{-\frac{\beta}{p}}\norm*{|x|^{-a}\nabla u}_{L^p(\Omega)}=C\lambda(p,\Omega)^{-\frac{\beta}{p}},
        \end{aligned}
    \end{equation*}
    where $\beta$ is a constant satisfying \eqref{qwe}. The remainder of this proof is analogous to that of Corollary \ref{ccc6}.
\end{proof}

\subsection{Stability in the critical point setting}\label{sec2.3}
This subsection is devoted to Theorem \ref{thm5} and Corollary \ref{cor}. The proofs rely on the following three lemmas. The first lemma presents some properties of the set $P_u$ defined in \eqref{def2}.
\begin{lemma}\label{lem3}
    Assume $a,b,p,q$ satisfy the relations in \eqref{ckn1}. Then there exists a positive number $\delta_0=\delta_0(a,b,p,q,n)$ such that for any function $u\in D_a^p(\R^n)$, $v\in\mathcal{M}_{p,a,b}^s$, and for any positive number $\delta<\delta_0$, if
    \begin{equation*}
        \norm*{u-v}_{D_a^p}\leq \delta,
    \end{equation*}
    then $P_u\neq \emptyset$ and
    \begin{equation}\label{sss1}
        \sup_{V\in P_u}\norm*{V-v}_{D_a^p}=o_\delta(1).
    \end{equation}
\end{lemma}
The second lemma contains six elementary inequalities for vectors and numbers.
\begin{lemma}\label{lem4}
    Let $x,y\in\R^n$ and $a,b\in\R$. The following inequalities hold.

    \noindent$(1)$ If $2<p\leq 3$, then there exists a constant $C$ such that
    \begin{equation}\label{ine1}
        \left||x+y|^{p-2}(x+y)\cdot y-|x|^{p-2}(x+y)\cdot y-(p-2)|x|^{p-4}(x\cdot y)^2\right|\leq C|y|^p.
    \end{equation}
    $(2)$ If $p>3$, then there exists a constant $C$ such that
    \begin{equation}\label{ine2}
        \left||x+y|^{p-2}(x+y)\cdot y-|x|^{p-2}(x+y)\cdot y-(p-2)|x|^{p-4}(x\cdot y)^2\right|\leq C(|y|^p+|x|^{p-3}|y|^3).
    \end{equation}
    $(3)$ If $2<p\leq 3$, then there exists a constant $C$ such that
    \begin{align}\label{ine3}
        \left||x+y|^{p-2}(x+y)\cdot y-|y|^p-|x|^{p-2}x\cdot y|\right|\leq C|x|^{p-2}|y|^2.
    \end{align}
    $(4)$ If $p>3$, then there exists a constant $C$ such that
    \begin{align}\label{ine4}
        \left||x+y|^{p-2}(x+y)\cdot y-|y|^p-|x|^{p-2}x\cdot y|\right|\leq C(|x|^{p-2}|y|^2+|x||y|^{p-1}).
    \end{align}
    $(5)$ If $2<q\leq 3$, then there exists a constant $C$ such that
    \begin{equation}\label{ine5}
        \left|(a+b)|a+b|^{q-2}-a|a|^{q-2}-(q-1)|a|^{q-2}b\right|\leq C|b|^{q-1}.
    \end{equation}
    $(6)$ If $q>3$, then there exists a constant $C$ such that
    \begin{equation}\label{ine6}
        \left|(a+b)|a+b|^{q-2}-a|a|^{q-2}-(q-1)|a|^{q-2}b\right|\leq C(|b|^{q-1}+|a|^{q-3}|b|^2).
    \end{equation}
\end{lemma}
The proofs of these two lemmas are quite similar to those of \cite[Lemma 4.1]{Fig1} and \cite[Lemma 3.2]{Fig}, respectively, and so are omitted.

The third lemma is a spectral gap inequality obtained by Deng and Tian  \cite[Proposition 4.6]{Den}. Before stating this lemma, let us give a definition about the tangent space and orthogonal conditions.
\begin{definition}\label{tan}
    Assume $a,b,p,q$ satsify the relations in \eqref{ckn1}. Fix a nonzero function $V\in\mathcal{M}_{p,a,b}$. If $(a,b)\neq (0,0)$, the tangent space $T_V\mathcal{M}_{p,a,b}$ is a 2-dimensional linear space spanned by
    \begin{equation}\label{tan1}
        V(x),\quad \frac{d}{d\lambda}\Big|_{\lambda=1}\left(\lambda^{\frac{n-p-pa}{p}}V(\lambda x)\right).
    \end{equation}
    If $a=b=0$, the tangent space $T_V\mathcal{M}_{p,0,0}$ is an $(n+2)$-dimensional linear space spanned by
    \begin{equation}\label{tan2}
        V(x),\quad \frac{d}{d\lambda}\Big|_{\lambda=1}\left(\lambda^{\frac{n-p-pa}{p}}V(\lambda x)\right),\quad \partial_{x_i}V(x),\quad i=1,\ldots,n.
    \end{equation}
    For any function $\rho\in D_a^p(\R^n)$, we say that $\rho$ is orthogonal to $T_V\mathcal{M}_{p,a,b}$ if
    \begin{equation*}
        \int_{\R^n}|x|^{-qb}|V|^{q-2}W\rho \,dx=0
    \end{equation*}
    holds for any $W\in T_V\mathcal{M}_{p,a,b}$.
\end{definition}

\begin{lemma}\label{lem5}
    Assume $a,b,p,q$ satisfy the relations in \eqref{ckn1}. Then there exists a positive constant $\tau$ such that, for any nonzero function $V\in\mathcal{M}^s_{p,a,b}$ and function $\rho\in D_a^p(\R^n)$ orthogonal to $T_V\mathcal{M}_{p,a,b}$, it holds that
    \begin{equation}\label{spec}
     \begin{aligned}
         &\int_{\R^n}|x|^{-pa}\left(|\nabla V|^{p-2}|\nabla \rho|^2+(p-2)|\nabla V|^{p-4}|\nabla V\cdot\nabla\rho|^2\right)dx\\
         \geq&\, (1+\tau)(q-1)\int_{\R^n}|x|^{-qb}|V|^{q-2}\rho^2\,dx.
     \end{aligned}
    \end{equation}
\end{lemma}
Now we can state our proofs.

\begin{proof}[Proof of Theorem \ref{thm5}]
    Choose a function $v\in\mathcal{M}^s_{p,a,b}$ such that $\norm*{u-v}_{D_a^p}\leq 2\delta$. From Lemma \ref{lem3}, we know that $\norm*{V-v}_{D_a^p}=o_\delta(1)$. Therefore,
    \begin{equation}\label{cri1}
        \norm*{u-V}_{D_a^p}=o_\delta(1).
    \end{equation}
    Next, take a number $\mu$ such that
    \begin{equation}\label{cri2}
        \mu\int_{\R^n}|x|^{-qb}|V|^q\,dx=\int_{\R^n}|x|^{-qb}|V|^{q-2}Vu\,dx.
    \end{equation}
    From \eqref{cri1}, it follows that $\mu=1+o_\delta(1)$. Set $\rho:=u-\mu V$. We have $\norm*{\rho}_{D_a^p}=o_\delta(1)$. Combining the definition \eqref{def2} with the equation \eqref{cri2} shows that $\rho$ is orthogonal to $T_V\mathcal{M}_{p,a,b}$. Therefore, we can apply \eqref{spec} in the subsequent arguments.

    Let us test the Euler-Lagrange equation \eqref{el} with $\rho$:
    \begin{equation}\label{cri3}
        \begin{aligned}
            \left\langle H(u),\rho\right\rangle=-\int_{\R^n}|x|^{-pa}|\nabla u|^{p-2}\nabla u\cdot\nabla\rho\,dx+\int_{\R^n}|x|^{-qb}|u|^{q-2}u\rho\,dx.
        \end{aligned}
    \end{equation}
    The left-hand side is bounded by
    \begin{equation}\label{cri4}
        \norm*{H(u)}_{D_a^{-p}}\norm*{\rho}_{D_a^p}.
    \end{equation}
    To apply Lemma \ref{lem4} to the right-hand side, we split our arguments into three cases, depending on the values of $p$ and $q$.
    \\[8pt]
    \emph{Case 1: $2<p<q\leq 3$.}

    Applying \eqref{ine1}, \eqref{ine3} with $x=\mu\nabla V$ and $y=\nabla\rho$, and \eqref{ine5} with $a=\mu V$ and $b=\rho$, we deduce
    \begin{align}\label{case1.1}
           \int_{\R^n}|x|^{-pa}|\nabla u|^{p-2}\nabla u\cdot\nabla\rho dx\geq&\,\mu^{p-2}\int_{\R^n}|x|^{-pa}|\nabla V|^{p-4}\left(|\nabla V|^{2}\nabla u\cdot \nabla\rho+(p-2)|\nabla V\cdot\nabla\rho|^2\right)dx\nonumber\\
           &-C\int_{\R^n}|x|^{-pa}|\nabla \rho|^p\,dx,
        \end{align}
    \begin{align}\label{case1.2}
            \int_{\R^n}|x|^{-pa}|\nabla u|^{p-2}\nabla u\cdot\nabla\rho\,dx\geq&\int_{\R^n}|x|^{-pa}|\nabla \rho|^p\,dx+\mu^{p-1}\int_{\R^n}|x|^{-pa}|\nabla V|^{p-2}\nabla V\cdot \nabla\rho\,dx\nonumber\\
            &-\,C\mu^{p-2}\int_{\R^n}|x|^{-pa}|\nabla V|^{p-2}|\nabla\rho|^2\,dx
        \end{align}
    and
    \begin{align}\label{case1.3}
        \int_{\R^n}|x|^{-qb}|u|^{q-2}u\rho\,dx\leq&\,\mu^{q-1}\int_{\R^n}|x|^{-qb}|V|^{q-2}V\rho\,dx+\mu^{q-2}(q-1)\int_{\R^n}|x|^{-qb}|V|^{q-2}\rho^2\,dx\nonumber\\
        &+C\int_{\R^n}|x|^{-qb}|\rho|^{q}\,dx.
    \end{align}
    Combining the estimates \eqref{spec}, \eqref{cri3}, \eqref{cri4}, \eqref{case1.1}, \eqref{case1.3} with the orthogonal conditions of $\rho$, we find that
    \begin{equation}\label{case1.4}
        \begin{aligned}
            &\left(\mu^{p-2}-\mu^{q-2}(1+\tau)^{-1}\right)\int_{\R^n}|x|^{-pa}\left(|\nabla V|^{p-2}|\nabla\rho|^2+(p-2)|\nabla V|^{p-4}|\nabla V\cdot\nabla\rho|^2\right)\,dx\\
            \leq&\,C\int_{\R^n}|x|^{-qb}|\rho|^{q}\,dx+C\int_{\R^n}|x|^{-pa}|\nabla \rho|^p\,dx+\norm*{H(u)}_{D_a^{-p}}\norm*{\rho}_{D_a^p}.
        \end{aligned}
    \end{equation}
    Since $\mu=1+o_\delta(1)$ and $\norm*{\rho}_{D_a^p}=o_\delta(1)$, by letting $\delta$ sufficiently small, we can obtain
    \begin{equation}\label{case1.5}
    \begin{aligned}
        c\int_{\R^n}|x|^{-pa}|\nabla V|^{p-2}|\nabla\rho|^2-C\norm*{\rho}^{p}_{D_a^p}
        \leq\norm*{H(u)}_{D_a^{-p}}\norm*{\rho}_{D_a^p}.
    \end{aligned}
    \end{equation}
    Similarly, combining the estimates \eqref{spec}, \eqref{cri3}, \eqref{cri4}, \eqref{case1.2}, \eqref{case1.3} with the orthogonal conditions, we derive
    \begin{equation}\label{case1.6}
        \begin{aligned}
            c\norm*{\rho}^{p}_{D_a^p}-C\int_{\R^n}|x|^{-pa}|\nabla V|^{p-2}|\nabla\rho|^2\,dx
        \leq\norm*{H(u)}_{D_a^{-p}}\norm*{\rho}_{D_a^p}.
        \end{aligned}
    \end{equation}
    \\[8pt]
    \emph{Case 2: $2<p\leq 3<q$.}

    Applying \eqref{ine6} with $a=\mu V$ and $b=\rho$, we get
    \begin{align}\label{case2.1}
        \int_{\R^n}|x|^{-qb}|u|^{q-2}u\rho\,dx\leq&\,\mu^{q-1}\int_{\R^n}|x|^{-qb}|V|^{q-2}V\rho\,dx+\mu^{q-2}(q-1)\int_{\R^n}|x|^{-qb}|V|^{q-2}\rho^2\,dx\nonumber\\
        &+C\int_{\R^n}|x|^{-qb}|\rho|^{q}\,dx+C\mu^{q-3}\int_{\R^n}|x|^{-qb}|V|^{q-3}|\rho|^{3}\,dx.
        \end{align}
    The only difference between \eqref{case1.3} and \eqref{case2.1} is the term
    \begin{equation*}
        \int_{\R^n}|x|^{-qb}|V|^{q-3}|\rho|^{3}\,dx,
    \end{equation*}
    which can be estimated by \eqref{ckn}, \eqref{spec} and the Young inequality:
    \begin{equation*}
        \int_{\R^n}|x|^{-qb}|V|^{q-3}|\rho|^{3}\,dx\leq \varepsilon\norm*{\rho}^{q}_{D_a^p}+C(\varepsilon)\int_{\R^n}|x|^{-pa}|\nabla V|^{p-2}|\nabla\rho|^2\,dx
    \end{equation*}
    and
    \begin{equation*}
        \int_{\R^n}|x|^{-qb}|V|^{q-3}|\rho|^{3}\,dx\leq \varepsilon\int_{\R^n}|x|^{-pa}|\nabla V|^{p-2}|\nabla\rho|^2\,dx+C(\varepsilon)\norm*{\rho}^{q}_{D_a^p}.
    \end{equation*}
    Here, $C(\varepsilon)$ is a constant dependent on $a,b,p,q,n,\varepsilon$. Therefore, it is easy to see that \eqref{case1.5} and \eqref{case1.6} still hold in this case.
    \\[8pt]
    \emph{Case 3: $3<p<q$.}

    Applying \eqref{ine2}, \eqref{ine4} with $x=\mu V$ and $y=\nabla \rho$, we find that
    \begin{align}\label{case3.1}
           &\int_{\R^n}|x|^{-pa}|\nabla u|^{p-2}\nabla u\cdot\nabla\rho\,dx\nonumber\\
           \geq&\,\mu^{p-2}\int_{\R^n}|x|^{-pa}\left(|\nabla V|^{p-2}\nabla u\cdot \nabla\rho+(p-2)|\nabla V|^{p-4}|\nabla V\cdot\nabla\rho|^2\right)\,dx\nonumber\\
           &-C\int_{\R^n}|x|^{-pa}|\nabla \rho|^p\,dx-C\mu^{p-3}\int_{\R^n}|x|^{-pa}|\nabla V|^{p-3}||\nabla\rho|^3\,dx,
    \end{align}
    and
    \begin{align}\label{case3.2}
            &\int_{\R^n}|x|^{-pa}|\nabla u|^{p-2}\nabla u\cdot\nabla\rho\,dx\nonumber\\
            \geq&\int_{\R^n}|x|^{-pa}|\nabla \rho|^p\,dx+\mu^{p-1}\int_{\R^n}|x|^{-pa}|\nabla V|^{p-2}\nabla V\cdot \nabla\rho\,dx\nonumber\\
            &-\,C\mu^{p-2}\int_{\R^n}|x|^{-pa}|\nabla V|^{p-2}|\nabla\rho|^2\,dx-C\mu\int_{\R^n}|x|^{-pa}|\nabla V||\nabla\rho|^{p-1}\,dx.
    \end{align}
    It follows from the Young inequality that the terms
    \begin{align}
        \int_{\R^n}|x|^{-pa}|\nabla V|^{p-3}|\nabla\rho|^3\,dx\nonumber
    \end{align}
    and
    \begin{align}
        \int_{\R^n}|x|^{-pa}|\nabla V||\nabla\rho|^{p-1}\,dx\nonumber
    \end{align}
    can simultaneously be controlled by
    \begin{equation*}
        \varepsilon\norm*{\rho}^{p}_{D_a^p}+C(\varepsilon)\int_{\R^n}|x|^{-pa}|\nabla V|^{p-2}|\nabla\rho|^2\,dx
    \end{equation*}
    and
    \begin{equation*}
        \varepsilon\int_{\R^n}|x|^{-pa}|\nabla V|^{p-2}|\nabla\rho|^2\,dx+C(\varepsilon)\norm*{\rho}^{p}_{D_a^p}.
    \end{equation*}
    Therefore, as discussed in \emph{Case 1}, we can obtain the estimates \eqref{case1.5} and \eqref{case1.6}.
\end{proof}

\begin{proof}[Proof of Corollary \ref{cor}]
     From Theorem \ref{thm5}, we have \eqref{sta6} and \eqref{sta7}. If $\rho=0$, it can be easily shown that  both \eqref{cor1} and \eqref{cor2} hold. In the following, we assume $\rho\neq 0$. Now, if the value of
    \begin{equation}\label{aaa}
        \frac{\norm*{\rho}^{p}_{D_a^p}}{\int_{\R^n}|x|^{-pa}|\nabla V|^{p-2}|\nabla\rho|^2\,dx}
    \end{equation}
    falls outside the interval $\left(\frac{c_1}{2C_1},\frac{2C_1}{c_1}\right)$, we can deduce
    \begin{equation}\label{aaa1}
        \norm*{H(u)}_{D_a^{-p}}\geq \frac{c_1}{2}\norm*{\rho}^{p-1}_{D_a^p}.
    \end{equation}
    If not, set $\eta=\left(\frac{c_1}{2C_1}\right)^{\frac{2}{p-2}}$. For any $0<t\leq\eta$, we can apply Theorem \ref{thm5} to $u_t:=tu+(1-t)V\,\in D_a^p(\R^n)$, $V\in P_{u_t}$, $\mu_t:=t\mu+1-t$, and $\rho_t:=u_t-\mu_tV=t\rho$. It is straightforward to check that
    \begin{equation*}
        \frac{\norm*{\rho_t}^{p}_{D_a^p}}{\int_{\R^n}|x|^{-pa}|\nabla V|^{p-2}|\nabla\rho_t|^2\,dx}<\frac{c_1}{2C_1}.
    \end{equation*}
    Thus, the estimate
    \begin{equation}\label{aaa2}
         \norm*{H(u_t)}_{D_a^{-p}}\geq \frac{c_1}{2}\norm*{\rho_t}^{p-1}_{D_a^p}
    \end{equation}
    holds for any $0<t\leq \eta$.

    It remains to derive \eqref{cor1} and \eqref{cor2} from \eqref{aaa1} and \eqref{aaa2}, respectively. Without loss of generality, we assume $\eqref{aaa1}$ holds here. From the Minkowski inequality, it suffices to prove
    \begin{equation}\label{zzz1}
        \begin{aligned}
            |\mu-1|\leq C\norm*{H(u)}_{D_a^{-p}}+C\norm*{\rho}_{D_a^p}.
        \end{aligned}
    \end{equation}
    Consider the following identity:
    \begin{align}\label{zzz2}
        \left(\mu^{p-1}-\mu^{q-1}\right)|x|^{-bq}V^{q-1}=&\,\div(|x|^{-pa}|\nabla u|^{p-2}\nabla u)-\mu^{p-1}\div(|x|^{-pa}|\nabla V|^{p-2}\nabla V)\nonumber\\
        &-\div(|x|^{-pa}|\nabla u|^{p-2}\nabla u)-|x|^{-qb}|u|^{q-2}u\nonumber\\
        &+|x|^{-qb}|u|^{q-2}u-\mu^{q-1}|x|^{-qb}|V|^{q-2}V.\nonumber
    \end{align}
    Testing this identity with the function $V$ and applying Lemma \ref{lem4} along with the H\"older inequality, we deduce
    \begin{equation*}
        |\mu^{p-1}-\mu^{q-1}|\leq C\norm*{H(u)}_{D_a^{-p}}+C\norm*{\rho}_{D_a^p}.
    \end{equation*}
    Since $\mu=1+o_\delta(1)$ and $p<q$, $|\mu-1|$ is comparable to $|\mu^{p-1}-\mu^{q-1}|$. Thus, \eqref{zzz1} holds.
\end{proof}

\vskip0.4in
\section{Improved Hardy-type inequalities}\label{sec3}
In this section, we focus on studying various improved Hardy-type inequalities. In each subsection, we address a specific type of Hardy-type inequality. An important tool we employ is a class of transformations. A prototype of these transformations is presented in Subsection \ref{sec2.1}, where it has been shown to be useful in deriving the gradient stability of the Caffarelli-Kohn-Nirenberg inequality. We would like to inform readers that we will omit some details in the following proofs when encountering computations and arguments similar to those presented in Subsection \ref{sec2.1}.

\subsection{The Hardy-Sobolev inequality}\label{sec3.1}
In this subsection, we consider the following Hardy-Sobolev inequality: for $1<p<n$, $0\leq a<\frac{n-p}{p}$, and $u\in D_a^p(\R^n)$,
\begin{align}\label{kk3}
    \int_{\R^n}|x|^{-ap}|\nabla u|^pdx\geq\left(\frac{n-p-pa}{p}\right)^p\int_{\R^n}|x|^{-(a+1)p}|u|^pdx.
\end{align}
The constant $\left(\frac{n-p-pa}{p}\right)^p$ is optimal but cannot be achieved by nontrivial functions. When $a=0$, \eqref{kk3} reduces to the classical Hardy inequality \eqref{kk1}. In \cite{Bre1}, for the first time in the literature, Br\'ezis and V\'azquez obtained improved Hardy inequalities on bounded domains (see \eqref{kk2}). Their pioneering work has received enormous attention, and we refer interested readers to the papers \cite{Abd,Adi0,Adi1,Barb0,Barb0.5,Barb,Bre1,Cuo,Fil,Gaz0.5,Gki,Psa,Vaz,Wan,Zog} and the monographs \cite{Bal0,Gho00,Kuf,Kuf1} for related topics. Here, we only mention some previous results that are closely related to ours.

First, we recall the estimate \eqref{kk2} obtained in \cite{Bre1}. In the particular case $p=2$, Br\'ezis and V\'azquez found the explicit stability constant:
\begin{align}\label{kk4}
    \int_{\Omega}|\nabla u|^2dx-\left(\frac{n-2}{2}\right)^2\int_{\Omega}\frac{|u|^2}{|x|^2}dx\geq z_0^2w_n^\frac{2}{n}|\Omega|^{-\frac{2}{n}}\int_{\Omega}|u|^2dx,
\end{align}
where $z_0=2.4048\cdots$ is the first zero of the Bessel function $J_0(z)$ and $w_n$ is the volume of the unit ball. Moreover, the constant $z_0^2w_n^2|\Omega|^{-\frac{2}{n}}$ is optimal when $\Omega$ is a ball, but it is not achieved. In \cite{Gaz0.5}, Gazzola, Grunau and Mitidieri extended \eqref{kk4} to the case $1<p<n$:
\begin{align}\label{kk5}
    \int_{\Omega}|\nabla u|^pdx-\left(\frac{n-p}{p}\right)^p\int_{\Omega}\frac{|u|^p}{|x|^p}dx\geq \Gamma_pw_n^\frac{p}{n}|\Omega|^{-\frac{p}{n}}\int_{\Omega}|u|^pdx.
\end{align}
The constant $\Gamma_p$ above has an explicit form, and in particular $\Gamma_2=z_0^2$. We refer to \cite[Theorem 1]{Gaz0.5} for the expression of $\Gamma_p$. We also note that, to the best of our knowledge, the optimality of $\Gamma_p\;(p\neq2)$ is still unknown; for the case $p=3$, some related results are obtained in \cite{Cuo}. In \cite{Fil}, by introducing suitable logarithmic terms, Filippas and Tertikas extended \eqref{kk2} to the critical case $q=\frac{2n}{n-2}$:
\begin{align}\label{kk6}
    \int_{\Omega}|\nabla u|^2dx-\left(\frac{n-2}{2}\right)^2\int_{\Omega}\frac{|u|^2}{|x|^2}dx\geq C_n(s)\left(\int_{\Omega}X_1^\frac{2n-2}{n-2}\left(s,\frac{|x|}{D}\right)|u|^\frac{2n}{n-2}dx\right)^\frac{n-2}{n},
\end{align}
where $s>0$, $D=\sup\limits_{x\in\Omega}|x|$, and $X_1(s,t):=(s-\ln{t})^{-1}$. Moreover, the exponent $\frac{2n-2}{n-2}$ is optimal and cannot be decreased. The optimal stability constant $C_n(s)$ was later found by Adimurthi, Filippas and Tertikas in \cite{Adi1}. Specfically, they showed that
\begin{align}\label{kk7}
    C_n(s)=\begin{cases}
        (n-2)^{-\frac{2n-2}{n}}S(2,0,0),\quad &s\geq\frac{1}{n-2},\\
        s^\frac{2n-2}{n}S(2,0,0),\quad &0<s<\frac{1}{n-2}.
    \end{cases}
\end{align}
Recall that $S(2,0,0)$ is the optimal Sobolev constant defined in \eqref{ckn2}. Moreover, they proved that $C_n(s)$ cannot be achieved. When $\Omega$ is the ball $B_R$ and the function $u$ is radially symmetric, from the results in \cite{Adi1,Zog}, \eqref{kk6} holds with $s=0$ and the best constant $C_{n,\text{rad}}=(n-2)^{-\frac{2n-2}{n}}S(2,0,0)$. Moreover, $C_{n,\text{rad}}$ can be achieved by the functions:
\begin{align}\label{kk8}
    u(x)=|x|^\frac{n-2}{2}\left(\mu^2+\nu^2\left(\ln{R}-\ln{|x|}\right)^{-\frac{2}{n-2}}\right)^{-\frac{n-2}{2}},\quad\mu\neq0,\;\nu\neq0.
\end{align}
Here, we also mention a related result given by Stubbe \cite{Stu}, which interpolates the Hardy inequality and the critical Sobolev inequality in the whole space $\R^n$: for any $0<\delta<1$,
\begin{align}\label{kk9}
     \int_{\R^n}|\nabla u|^2dx-\left(\frac{n-2}{2}\right)^2\delta\int_{\R^n}\frac{|u|^2}{|x|^2}dx\geq (1-\delta)^\frac{n-1}{n}S(2,0,0)\left(\int_{\R^n}|u|^\frac{2n}{n-2}dx\right)^\frac{n-2}{n}.
\end{align}
The constant $(1-\delta)^\frac{n-1}{n}S(2,0,0)$ is optimal and can be achieved by the function:
\begin{align}\label{kk10}
    u(x)=|x|^{-\frac{n-2}{2}(1-\sqrt{1-\delta})}\left(1+|x|^{2\sqrt{1-\delta}}\right)^{-\frac{n-2}{2}}.
\end{align}

Now, let us state our results, which extend the above estimates \eqref{kk4}, \eqref{kk5}, \eqref{kk6} and \eqref{kk9} to the Hardy-Sobolev inequality \eqref{kk3}.
\begin{theorem}\label{thm1.1}
    For any domain $\Omega\subset\R^n$ with finite volume, $1<p<n$, $0<a<\frac{n-p}{p}$, and for every $u\in D_a^p(\Omega)$, we have
    \begin{align}\label{kk11}
        \int_{\Omega}\frac{|\nabla u|^p}{|x|^{ap}}dx-\left(\frac{n-p-pa}{p}\right)^p\int_{\Omega}\frac{|u|^p}{|x|^{(a+1)p}}dx\geq& \;\Gamma_{p,a}\left(\frac{w_n}{|\Omega^{(p,a)}|}\right)^\frac{p}{n} \int_{\Omega}\frac{|u|^p}{|x|^{\frac{apn}{n-p}}}dx\nonumber\\
        \geq&\;\Gamma_{p,a}\left(\frac{w_n}{|\Omega|}\right)^\frac{p(n-p-pa)}{n(n-p)} \int_{\Omega}\frac{|u|^p}{|x|^{\frac{apn}{n-p}}}dx,
    \end{align}
    where $\Gamma_{p,a}=\left(\frac{n-p-pa}{n-p}\right)^p\Gamma_p$ (recall \eqref{kk5}) and $\Omega^{(p,a)}:=\left\{x\;\big|\;|x|^{\frac{ap}{n-p-pa}}x\in\Omega\right\}$. Moreover, when $\Omega$ is a ball centered at the origin and $p=2$, the constant $\Gamma_{2,a}$ is optimal but cannot be achieved.
\end{theorem}
\begin{theorem}\label{thm1.2}
    For any bounded domain $0\in\Omega\subset\R^n$, where $n\geq 3$, $0<a<\frac{n-2}{2}$, $s>0$, and $D:=\sup\limits_{x\in\Omega}|x|$, the following inequality holds for every $u\in D_0^2(\Omega)$:
    \begin{align}\label{kk12}
        \int_{\Omega}\frac{|\nabla u|^2}{|x|^{2a}}dx-\left(\frac{n-2-2a}{2}\right)^2\int_{\Omega}\frac{|u|^2}{|x|^{2(a+1)}}dx\geq C_n(s)\left(\int_{\Omega}X_1^\frac{2n-2}{n-2}\left(s,\frac{|x|}{D}\right)\frac{|u|^\frac{2n}{n-2}}{|x|^\frac{2an}{n-2}}dx\right)^\frac{n-2}{n},
    \end{align}
    where $C_n(s)$ is the constant given in \eqref{kk7}. Moreover, $C_n(s)$ is optimal and cannot be achieved. Additionally, if $\Omega$ is the ball $B_R$, $s=0$ and $u$ is radially symmetric, then \eqref{kk12} holds with the best constant $C_{\text{rad}}=(n-2)^{-\frac{2n-2}{n}}S(2,0,0)$, and $C_{\text{rad}}$ can be attained by the functions:
    \begin{align}
        u(x)=|x|^\frac{n-2-2a}{2}\left(\mu^2+\nu^2\left(\ln{R}-\ln{|x|}\right)^{-\frac{2}{n-2}}\right)^{-\frac{n-2}{2}},\quad\mu\neq0,\;\nu\neq0.\nonumber
    \end{align}
\end{theorem}
\begin{theorem}\label{thm1.3}
    Assume $n\geq 3$, $0<a<\frac{n-2}{2}$, $0<\delta<1$, and $u\in D_a^2(\R^n)$. Then, the following inequality holds:
    \begin{align}\label{kk13}
        \int_{\R^n}\frac{|\nabla u|^2}{|x|^{2a}}dx-\left(\frac{n-2-2a}{2}\right)^2\delta\int_{\R^n}\frac{|u|^2}{|x|^{2(a+1)}}dx\geq (1-\delta)^\frac{n-1}{n}S(2,a,a)\left(\int_{\R^n}\frac{|u|^\frac{2n}{n-2}}{|x|^\frac{2an}{n-2}}dx\right)^\frac{n-2}{n}.
    \end{align}
    Moreover, the constant $(1-\delta)^\frac{n-1}{n}S(2,a,a)$ is sharp and can be attained by the function:
    \begin{align}
        u(x)=|x|^{-\frac{n-2-2a}{2}(1-\sqrt{1-\delta})}\left(1+|x|^{2\sqrt{1-\delta}}\right)^{-\frac{n-2-2a}{2}}.\nonumber
    \end{align}
\end{theorem}
\begin{proof}[Proof of the above three theorems]
    Theorem \ref{thm1.1}, Theorem \ref{thm1.3}, and the radial case in Theorem \ref{thm1.2} can be derived by applying a class of transformations: $\Bar{u}(r\theta)=k^\frac{1}{p}u(r^k\theta)$, where $1<p<n$, $0<a<\frac{n-p}{p}$, and $k=\frac{n-p}{n-p-pa}$, using \eqref{kk4}-\eqref{kk10}, and following a similar reasoning as in the proofs of Theorem \ref{thm1} and Theorem \ref{thm2}. It is noteworthy that in the proof of Theorem \ref{thm1.1}, to address the term $|\Omega|$, the following observation is useful: if $|\Omega|=|B_t|$ for some $t>0$, then $|\Omega^{(p,a)}|\leq |B_t^{(p,a)}|$, with equality holding precisely when $\Omega=B_t$ is a ball.

\vskip0.12in
    Next, we address the general case in Theorem \ref{thm1.2}. Following the arguments in \cite{Adi1}, we define $v(x)=|x|^\frac{n-2-2a}{2}u(x)$. By density, we can assume $u\in C_c^\infty(\Omega\backslash\{0\})$. A direct computation yields that the inequality \eqref{kk12} is equivalent to the following Hardy-Sobolev-type inequality:
    \begin{align}
        \int_{\Omega}\frac{|\nabla v|^2}{|x|^{(n-2)}}dx\geq C_n(s)\left(\int_{\Omega}X_1^\frac{2n-2}{n-2}\left(s,\frac{|x|}{D}\right)\frac{|v|^\frac{2n}{n-2}}{|x|^n}dx\right)^\frac{n-2}{n},\nonumber
    \end{align}
    which has been thoroughly studied in \cite[Theorem A']{Adi1}. Thus, our proof is complete.
\end{proof}

Finally, we note that compared to the ground state transformation $v(x)=|x|^\frac{n-2-2a}{2}u(x)$, the transformation $u\rightarrow\Bar{u}(r\theta):=k^\frac{1}{2}u(r^k\theta)$, with $k=\frac{n-2}{n-2-2a}$, provides an alternative method to derive \eqref{kk12} from \eqref{kk6}. However, this transformation results in a non-optimal stability constant when $s$ is small. The primary reason for this is that when $s$ is small, the sharpness of $C_n(s)$ is achieved by testing suitable non-radial functions, which behave poorly under the transformation $u\rightarrow\Bar{u}$.

Nevertheless, for the case $p\neq 2$, where applying the ground state transformation is somewhat complicated, the transformation $u\rightarrow\Bar{u}$ remains effective and simpler. To illustrate it, we consider the following result obtained in \cite{Gki}, which extends the estimate \eqref{kk6}: assume $0\in\Omega\subset\R^n$ is a bounded domain, $n\geq2$, $1<p<n$, $m\in\N$, and $u\in D_0^p(\Omega)$. Then there exist constants $B=B(n,p)\geq 1$ and $C=C(n,p)>0$ such that
\begin{align}\label{kk14}
    \int_{\Omega}|\nabla u|^pdx-\left(\frac{n-p}{p}\right)^p\int_{\Omega}\frac{|u|^p}{|x|^p}dx&-\frac{p-1}{2p}\left(\frac{n-p}{p}\right)^{p-2}\sum_{i=1}^m\int_{\Omega}\frac{|u|^p}{|x|^p}Y_i^2\left(\frac{|x|}{D}\right)dx\nonumber\\
    &\geq C\left(\int_{\Omega}|u|^\frac{np}{n-p}Y_{m+1}^{\frac{2n-p}{n-p}}\left(\frac{|x|}{D}\right)dx\right)^\frac{n-p}{n},
\end{align}
where $D:=B\sup\limits_{x\in\Omega}|x|$, $Y_j(t):=\prod\limits_{i=1}^jX_j(t)$, $X_j(t):=X_1(X_{j-1}(t))$, and $X_1(t):=(1-\ln{t})^{-1}$. Moreover, the weight function $Y_{m+1}^{\frac{2n-p}{n-p}}$ is optimal in the sense that the exponent $\frac{2n-p}{n-p}$ on $X_{m+1}$ cannot be decreased.

Now, by applying the transformation $u\rightarrow\Bar{u}(r\theta)=k^\frac{1}{p}u(r^k\theta)$ with $1<p<n$, $0<a<\frac{n-p}{p}$, and $k=\frac{n-p}{n-p-pa}$, and reasoning similarly as in the proofs of Theorems \ref{thm1} and \ref{thm2}, we can obtain the following estimate \eqref{kk15}.
\begin{theorem}\label{thm1.4}
    Assume $0\in\Omega\subset\R^n$ is a bounded domain, with $n\geq2$, $1<p<n$, $0<a<\frac{n-p}{p}$, $m\in\N$, and $u\in D_a^p(\Omega)$. Then there exist constants $B=B(n,a,p)\geq 1$ and $C=C(n,a,p)>0$ such that
    \begin{align}\label{kk15}
        &\;\int_{\Omega}\frac{|\nabla u|^p}{|x|^{ap}}dx-\left(\frac{n-p-pa}{p}\right)^p\int_{\Omega}\frac{|u|^p}{|x|^{p(a+1)}}dx\\
        \geq&\;\frac{p-1}{2p}\left(\frac{n-p-pa}{p}\right)^{p-2}\sum_{i=1}^m\int_{\Omega}\frac{|u|^p}{|x|^{p(a+1)}}Y_i^2\left(\frac{|x|}{D}\right)dx+ C\left(\int_{\Omega}\frac{|u|^\frac{np}{n-p}}{|x|^\frac{apn}{n-p}}Y_{m+1}^{\frac{2n-p}{n-p}}\left(\frac{|x|}{D}\right)dx\right)^\frac{n-p}{n},\nonumber
    \end{align}
    where $D$ and $Y_i$ are defined as above. Moreover, the weight function $Y_{m+1}^{\frac{2n-p}{n-p}}$ is optimal in the sense that the exponent $\frac{2n-p}{n-p}$ on $X_{m+1}$ cannot be decreased.
\end{theorem}
\begin{proof}
    For the derivation of \eqref{kk15}, here we illustrate how to handle the term $Y_i$. We need the following two general inequalities: Assume $k>1$ and $0<t\leq1$ Then there exists a positive constant $C_0=C_0(k)$ such that for any $C\geq C_0$, we have
    \begin{align}
        X_1(t)\geq kX_1(C^{-1}t^k),\quad X_i(t)\geq X_i(C^{-1}t^k) \;\;\forall\; 2\leq i\leq m.\nonumber
    \end{align}
    The first inequality is straightforward, and one can see that $C_0=e^{k-1}$ is sufficient. For the second inequality, we can choose $C_0=1$, as $X_i$ is an increasing function on $[0,1]$. Thus, in conclusion, we can take $C_0=e^{k-1}$. Thanks to these two estimates, we can obtain the constant $\frac{p-1}{2p}\left(\frac{n-p-pa}{p}\right)^{p-2}$ in \eqref{kk15}. Moreover, since $C_0$ is independent of $m$, the constants $B$ and $C$ in the statement of Theorem \ref{thm1.4} are also independent of $m$.
    
    It remains to show the optimality of the exponent $\frac{2n-p}{n-p}$ on $X_{m+1}$. We use similar ideas as in \cite{Gki}. For simplicity, we define
    \begin{align}
        I_m[u]:=&\;\int_{\Omega}\frac{|\nabla u|^p}{|x|^{ap}}dx-\left(\frac{n-p-pa}{p}\right)^p\int_{\Omega}\frac{|u|^p}{|x|^{p(a+1)}}dx\nonumber\\
        &-\frac{p-1}{2p}\left(\frac{n-p-pa}{p}\right)^{p-2}\sum_{i=1}^m\int_{\Omega}\frac{|u|^p}{|x|^{p(a+1)}}Y_i^2\left(\frac{|x|}{D}\right)dx,\nonumber
    \end{align}
    and we will omit the term $\frac{|x|}{D}$ in the following. If for some $0<\varepsilon<1$, there exist constants $B\geq1$ and $C>0$ such that
    \begin{align}
        I_m[u]\geq C\left(\int_{\Omega}\frac{|u|^\frac{np}{n-p}}{|x|^\frac{apn}{n-p}}Y_{m}^{\frac{2n-p}{n-p}}X_{m+1}^{\frac{(2n-p)\varepsilon}{n-p}}dx\right)^\frac{n-p}{n}\nonumber
    \end{align}
    holds for every $u\in D_a^p(\Omega)$, then from the H\"older inequality, for any $\gamma>0$, we have
    \begin{align}
        \int_{\Omega}\frac{|u|^p}{|x|^{p(a+1)}}Y_m^2X_{m+1}^\gamma dx
        =&\;\int_{\Omega}\left(|x|^{-p}Y_m^{\frac{p}{n}}X_{m+1}^{\gamma-\frac{(2n-p)\varepsilon}{n}}\right)\left(\frac{|u|^p}{|x|^{pa}}Y_m^\frac{2n-p}{n}X_{m+1}^\frac{(2n-p)\varepsilon}{n}\right)dx\nonumber\\
        \leq&\;C\left(\int_{\Omega}|x|^{-n}Y_mX_{m+1}^{\frac{n\gamma-(2n-p)\varepsilon}{p}}dx\right)^\frac{p}{n}I_m[u].\nonumber
    \end{align}
    From \cite[formula (3.8)]{Barb0.5}, we know that the integral
    \begin{align}
        \int_{\Omega}|x|^{-n}Y_mX_{m+1}^{\frac{n\gamma-(2n-p)\varepsilon}{p}}dx\nonumber
    \end{align}
    is finite if and only if
    \begin{align}
        \frac{n\gamma-(2n-p)\varepsilon}{p}>1,\nonumber
    \end{align}
    which is equivalent to
    \begin{align}
        \gamma>2-\frac{(2n-p)(1-\varepsilon)}{n}.\nonumber
    \end{align}
    Since $\varepsilon<1$, we can take a number $\gamma<2$ such that
    \begin{align}
        \int_{\Omega}\frac{|u|^p}{|x|^{p(a+1)}}Y_m^2X_{m+1}^\gamma dx\leq CI_m[u].\nonumber
    \end{align}
    However, this contradicts \cite[Theorem 2]{Barb0}. Therefore, the optimality is proved.
\end{proof}

\subsection{The \texorpdfstring{$L^p$}.-logarithmic Sobolev inequality with weights}\label{sec3.2}
In this subsection, we consider the following Euclidean $L^p$-logarithmic Sobolev inequality: for $n\geq1$, $p>1$, and $u\in D_0^p(\R^n)$ with $\int_{\R^n}|u|^pdx=1$, we have
\begin{align}\label{qq1}
    \int_{\R^n}|u|^p\ln{|u|^p}dx\leq \frac{n}{p}\ln\left(\LL_p\int_{\R^n}|\nabla u|^pdx\right).
\end{align}
The sharp constant $\LL_p$ has the form:
\begin{align}\label{qq2}
    \frac{p}{n}\left(\frac{p-1}{e}\right)^{p-1}\left(\Gamma\left(\frac{(p-1)n}{p}+1\right)w_n\right)^{-\frac{p}{n}},
\end{align}
where $w_n$ is the volume of the unit ball. All the extremizers of \eqref{qq1} are given by the following family of Gaussians:
\begin{align}\label{qq3}
    u_{\lambda,x_0}(x)=\lambda e^{-\alpha_\lambda |x+x_0|^{\frac{p}{p-1}}},
\end{align}
where $\lambda\neq0$, $x_0\in\R^n$, and $\alpha_\lambda$ is a positive constant that depends on $\lambda$ such that $\int_{\R^n}|u_{\lambda,x_0}|^pdx=1$. In the case $p=2$ and $n\geq 3$, \eqref{qq1} was obtained by Weissler \cite{Weis}, and an equivalent form for the Gaussian measure was derived by Gross \cite{Gro}. The optimizers in this case were classified by Carlen \cite{Car}. In the case $1<p<n$, \eqref{qq1} along with its extremizers was obtained by Del Pino and Dolbeault \cite{del0}. In the remaining case $p\geq n\geq1$, \eqref{qq1} was obtained by Gentil \cite{Gen} and Fujita \cite{Fuj}. The classification of extremal functions was derived recently by Balogh, Don and Krist\'aly \cite{Bal}.

Our first result establishes weighted versions of the inequality \eqref{qq1} and provides a complete classification of the corresponding extremizers.
\begin{theorem}\label{thm2.1}
    Assume $n\geq 1$, $p>1$, $p\neq n$, and $0<a<\left|\frac{n-p}{p}\right|$. Then for any $u\in D_a^p(\R^n)$ with $\int_{\R^n}|x|^{\frac{-apn}{|n-p|}}|u|^pdx=1$, we have
\begin{align}\label{qq4}
    \int_{\R^n}|x|^{\frac{-apn}{|n-p|}}|u|^p\ln{|u|^p}dx\leq \frac{n}{p}\ln\left(\LL_{p,a}\int_{\R^n}|x|^{-ap\chi(n-p)}|\nabla u|^pdx\right),
\end{align}
where $\chi(t):=\frac{t}{|t|}$ for any $t\neq 0$. The sharp constant $\LL_{p,a}$ is given by $\LL_p\left(\frac{|n-p|}{|n-p|-ap}\right)^{\frac{n-1}{n}p}$, and all the extremizers are given by the Gaussians:
\begin{align}
    u(x)=\lambda e^{-\alpha_\lambda |x|^{\frac{p(|n-p|-ap)}{(p-1)|n-p|}}},\nonumber
\end{align}
where $\lambda\neq 0$ and $\alpha_\lambda$ is a positive normalization factor that depends on $\lambda$.
\end{theorem}
\begin{theorem}\label{thm2.2}
    Assume $p=n>1$ and $0<s<1$. Then for any function $u\in D_0^n(\R^n)$ with $\int_{\R^n}|x|^{-ns}|u|^ndx=1$, we have
\begin{align}\label{qq5}
    \int_{\R^n}|x|^{-ns}|u|^n\ln{|u|^n}dx\leq \ln\left(\LL_{n,s}\int_{\R^n}|\nabla u|^ndx\right).
\end{align}
The sharp constant $\LL_{n,s}$ is given by $\LL_n(1-s)^{1-n}$, and all the extremizers are given by the Gaussians:
\begin{align}
    u(x)=\lambda e^{-\alpha_\lambda|x|^{\frac{n(1-s)}{n-1}}},\nonumber
\end{align}
where $\lambda\neq 0$ and $\alpha_\lambda$ is a positive normalization factor that depends on $\lambda$.
\end{theorem}
\begin{remark}
    In the case $1<p\leq n$, the inequalities \eqref{qq4} and \eqref{qq5} have been obtained by Lam and Lu \cite{Lam1}. Their approach is mainly based on the endpoint differentiation method, which is now a standard technique for deriving logarithmic-type inequalities. However, this method does not provide much insight into the extremizers. Our approach is not only simpler but also offers a complete classification of the optimizers.
\end{remark}
\begin{proof}[Proof of the above two theorems]
    For Theorem \ref{thm2.1}, we consider the transformation $u\rightarrow\Bar{u}(r\theta):=k^{\frac{1}{p}}u(r^k\theta)$ with $k=\frac{|n-p|}{|n-p|-ap}$. For Theorem \ref{thm2.2}, we consider the transformation $u\rightarrow\Bar{u}(r\theta):=k^{\frac{1}{n}}u(r^k\theta)$ with $k=\frac{1}{1-s}$. Based on the previous results \eqref{qq1}-\eqref{qq3}, the remaining arguments are similar to those in the proof of Theorem \ref{thm1} and are thus omitted.
\end{proof}

Next, we consider the following result obtained recently by Balogh, Don and Krist\'aly \cite{Bal}, which establishes a weighted version of the inequality \eqref{qq1} involving a log-concave homogeneous weight on an open convex cone.

Assume $n\geq1$, $p>1$, $E\subset\R^n$ is an open convex cone, and $w:E\rightarrow\R_+$ is a log-concave $C^1$ homogeneous function with degree $\tau\geq0$. Then for every function $u\in D_0^p(w;E)$ with $\int_E|u|^pwdx=1$, we have
\begin{align}\label{qq6}
    \int_{E}|u|^p\ln{|u|^p}wdx\leq \frac{n+\tau}{p}\ln\left(\LL_{w,p}\int_{E}|\nabla u|^pwdx\right).
\end{align}
(Here we say $u\in D_0^p(w;E)$ if both $u$ and $\nabla u$ belong to the weighted Lebesgue space $L^p(w;E)$.) The sharp constant $\LL_{w,p}$ has the form:
\begin{align}\label{qq7}
    \frac{p}{n+\tau}\left(\frac{p-1}{e}\right)^{p-1}\left(\Gamma\left(\frac{(p-1)(n+\tau)}{p}+1\right)\int_{B_1\cap E}wdx\right)^{-\frac{p}{n+\tau}},
\end{align}
and all the optimizers are given by the Gaussians:
\begin{align}\label{qq8}
    u_{\lambda,x_0}(x)=\lambda e^{-\alpha_\lambda |x+x_0|^{\frac{p}{p-1}}},
\end{align}
where $\lambda\neq0$, $x_0\in\R^n$ such that
\begin{align}\label{qq9}
    \begin{cases}
        x_0\in\partial E\cap-\partial E\;\text{and}\;w(x+x_0)=w(x)\;\;\forall\;x\in E,\quad\text{if }\tau>0,\\
        x_0\in\Bar{E}\cap-\Bar{E},\quad\text{if }\tau=0,
    \end{cases}
\end{align}
and $\alpha_\lambda$ is a positive constant that depends on $\lambda$ such that $\int_{E}|u_{\lambda,x_0}|^pwdx=1$. It is worth mentioning that from \cite[Proposition 2.1]{Bal}, $\tau=0$ if and only $w$ is a constant in $E$, and in this case, \eqref{qq6} also extends \eqref{qq1} to convex cones.

Our second result extends Theorem \ref{thm2.1} and Theorem \ref{thm2.2} to the inequality \eqref{qq6}.
\begin{theorem}\label{thm2.3}
    Let $E\subset\R^n\;(n\geq1)$ be an open convex cone and $w:E\rightarrow\R_+$ be a log-concave $C^1$ homogeneous function with degree $\tau\geq0$. Assume $p>1$, $p\neq n+\tau$, and $0<a<\left|\frac{n+\tau-p}{p}\right|$. Then for every function $u\in D_a^p(w;E)$ with $\int_E|x|^{-\frac{ap(n+\tau)}{|n+\tau-p|}}|u|^pwdx=1$, we have
\begin{align}\label{qq10}
    \int_{E}|x|^{-\frac{ap(n+\tau)}{|n+\tau-p|}}|u|^p\ln{|u|^p}wdx\leq \frac{n+\tau}{p}\ln\left(\LL_{w,p,a}\int_{E}|x|^{-ap\chi(n+\tau-p)}|\nabla u|^pwdx\right).
\end{align}
(Here we say $u\in D_a^p(w;E)$ if both $|x|^{-\frac{a(n+\tau)}{|n+\tau-p|}}u$ and $|x|^{-a\chi(n+\tau-p)}\nabla u$ belong to the weighted Lebesgue space $L^p(w;E)$.) The sharp constant $\LL_{w,p,a}$ equals $\LL_{w,p}\left(\frac{|n+\tau-p|}{|n+\tau-p|-ap}\right)^{\frac{n+\tau-1}{n+\tau}p}$, and all the extremizers are given by the Gaussians:
\begin{align}
    u(x)=\lambda e^{-\alpha_\lambda |x|^{\frac{p(|n+\tau-p|-ap)}{(p-1)|n+\tau-p|}}},\nonumber
\end{align}
where $\lambda\neq0$ and $\alpha_\lambda$ is a positive normalization factor that depends on $\lambda$.
\end{theorem}
\begin{theorem}\label{thm2.4}
    Let $E\subset\R^n\;(n\geq1)$ be an open convex cone and $w:E\rightarrow\R_+$ be a log-concave $C^1$ homogeneous function with degree $\tau\geq0$. Assume $p=n+\tau>1$ and $0<s<1$. Then for every function $u\in D^{n+\tau}_s(w;E)$ with $\int_E|x|^{-(n+\tau)s}|u|^{n+\tau}wdx=1$, we have
\begin{align}\label{qq11}
    \int_{E}|x|^{-(n+\tau)s}|u|^{n+\tau}\ln{|u|^{n+\tau}}wdx\leq \ln\left(\LL_{w,n+\tau,s}\int_{E}|\nabla u|^{n+\tau}wdx\right).
\end{align}
(Here we say $u\in D^{n+\tau}_s(w;E)$ if both $|x|^{s}u$ and $\nabla u$ belong to the weighted Lebesgue space $L^{n+\tau}(w;E)$.) The sharp constant $\LL_{w,n+\tau,s}$ equals $\LL_{w,n+\tau}(1-s)^{1-n-\tau}$, and all the extremizers are given by the Gaussians:
\begin{align}
    u(x)=\lambda e^{-\alpha_\lambda|x|^{\frac{(n+\tau)(1-s)}{n+\tau-1}}},\nonumber
\end{align}
where $\lambda\neq 0$ and $\alpha_\lambda$ is a positive normalization factor that depends on $\lambda$.
\end{theorem}
\begin{proof}[Proof of the above two theorems]
    For Theorem \ref{thm2.3}, we consider the transformation $u\rightarrow\Bar{u}(r\theta):=k^{\frac{1}{p}}u(r^k\theta)$ with $k=\frac{|n+\tau-p|}{|n+\tau-p|-ap}$. For Theorem \ref{thm2.4}, we consider the transformation $u\rightarrow\Bar{u}(r\theta):=k^{\frac{1}{n+\tau}}u(r^k\theta)$ with $k=\frac{1}{1-s}$. Note that under these transformations, the underlying domain $E$ is preserved, and the homogeneous weight $w$ behaves well: $w(r^k\theta)=r^kw(\theta)$. Therefore, based on the results \eqref{qq6}-\eqref{qq9}, we can argue similarly as in the proof of Theorem \ref{thm1} to obtain \eqref{qq10}, \eqref{qq11}, as well as their optimizers.
\end{proof}

Finally, we investigate potential stability estimates for the inequalities \eqref{qq4} and \eqref{qq5}. As in the proof of Theorem \ref{thm3}, for every stability result of the standard $L^p$-logarithmic Sobolev inequality \eqref{qq1}, it is possible to extend it to the weighted inequalities \eqref{qq4} and \eqref{qq5}, at least for centrally symmetric functions. To illustrate this, we consider the following stability estimate obtained by by Dolbeault, Esteban, Figalli, Frank and Loss \cite{Dol}. We also refer interested readers to \cite{Dol} for a detailed survey on the stability results of the logarithmic Sobolev inequality. 

There exists a positive constant $\beta>0$ such that for any $n\in\N_+$ and $u\in D_0^2(\R^n)$ with $\int_{\R^n}|u|^2dx=1$, it holds that
\begin{align}\label{qq12}
    \frac{n}{2}\ln\left(\LL_2\int_{\R^n}|\nabla u|^2dx\right)-\int_{\R^n}|u|^2\ln{|u|^2}dx\geq \beta\inf_{\lambda\neq 0,x_0\in\R^n}\int_{\R^n}\left|u-\lambda e^{-\alpha_\lambda |x+x_0|^2} \right|^2dx.
\end{align}
Here, $\alpha_\lambda$ is a positive constant such that $\int_{\R^n}\lambda^2 e^{-2\alpha_\lambda |x|^2}dx=1$. We remark that although the estimate \eqref{qq12} appears slightly stronger than \cite[Corollary 4.4]{Dol}, they are indeed equivalent. This can be seen by following similar approaches as in, for example, the proof of \cite[Theorem 3.3]{Caz1}.

Our final result extends the stability estimate \eqref{qq12} to \eqref{qq4} and \eqref{qq5}. For simplicity, we only state the case $n\geq3$, while the remaining cases $n=1,2$ can be similarly formulated.
\begin{theorem}\label{thm2.5}
    There exists a positive constant $\beta'>0$ such that for any $n\geq 3$, $0<a<\frac{n-2}{2}$, and centrally symmetric function $u\in D_a^2(\R^n)$ with $\int_{\R^n}|x|^{-\frac{2an}{n-2}}|u|^2dx=1$, it holds that
\begin{align}\label{qq13}
    \frac{n}{2}\ln\left(\LL_{2,a}\int_{\R^n}\frac{|\nabla u|^2}{|x|^{2a}}dx\right)-\int_{\R^n}\frac{|u|^2\ln{|u|^2}}{|x|^{\frac{2an}{n-2}}}dx\geq \beta'\inf_{\lambda\neq 0}\int_{\R^n}\frac{\left|u-\lambda e^{-\alpha_\lambda |x|^\frac{2(n-2-2a)}{n-2}} \right|^2}{|x|^{\frac{2an}{n-2}}}dx.
\end{align}
Here, $\alpha_\lambda$ is a positive constant such that $\int_{\R^n}|x|^{-\frac{2an}{n-2}}\lambda^2 e^{-2\alpha_\lambda |x|^\frac{2(n-2-2a)}{n-2}}dx=1$.
\end{theorem}
\begin{proof}
    Consider the transformation $u\rightarrow\Bar{u}(r\theta):=k^{\frac{1}{2}}u(r^k\theta)$ with $k=\frac{n-2}{n-2-2a}$. The estimate \eqref{qq13} can be derived from \eqref{qq12} using similar arguments as in the proof of Theorem \ref{thm3}, with a bit more care regarding the stability constant. Specifically, to ensure that the stability constant $\beta'$ is independent of the dimension $n$ and the constant $a$, we only need to show the following: there exists a positive constant $C$ such that for any $n\in\N_+$, $\alpha>0$, and $x_0\in\R^n$, it holds that
    \begin{align}\label{qq14}
        \int_{\R^n}\left|e^{-\alpha |x|^2}- e^{-\alpha |x+x_0|^2} \right|^2dx\leq C\int_{\R^n}\left|e^{-\alpha |x-x_0|^2}- e^{-\alpha |x+x_0|^2} \right|^2dx.
    \end{align}
    Up to a suitable scaling and rotation, we can assume $\alpha=1$ and $x_0=(|x_0|,0,\cdots,0)$. Now, \eqref{qq14} is equivalent to
    \begin{align}
        \int_{\R^{n-1}}e^{-|x'|^2}\int_{\R}\left(\left(e^{-x_1^2}- e^{-|x_1+|x_0||^2} \right)^2-C\left(e^{-|x_1-|x_0||^2}- e^{-|x_1+|x_0||^2} \right)^2\right)dx_1dx'\leq 0,\nonumber
    \end{align}
    where $x'=(x_2,\cdots,x_n)$. It suffices to choose $C$ such that \eqref{qq14} holds for $n=1$ and $\alpha=1$. The existence of such a $C$ follows directly from the mean value theorem.
\end{proof}

\subsection{The logarithmic Hardy inequality}\label{sec3.3}
In this subsection, we consider the following logarithmic Hardy inequality: assume $n\geq1$, $a<\frac{n-2}{2}$, $\gamma\geq \frac{n}{4}$, and $\gamma>\frac{1}{2}$ if $n=2$. Then, for every function $u\in D_a^2(\R^n)$ satisfying $\int_{\R^n}\frac{|u|^2}{|x|^{2(a+1)}}dx=1$, the following inequality holds:
\begin{align}\label{uu1}
    \int_{\R^n}\frac{|u|^2}{|x|^{2(a+1)}}\ln{\left(|x|^{n-2-2a}|u|^2\right)}dx\leq 2\gamma\ln\left(\C_{a}\int_{\R^n}\frac{|\nabla u|^2}{|x|^{2a}}dx\right).
\end{align}
Here, $\C_a$ denotes the sharp constant. The inequality \eqref{uu1} was obtained by Del Pino, Dolbeault, Filippas and Tertikas \cite{del1} using the endpoint differentiation technique applied to a special family of Caffarelli-Kohn-Nirenberg inequalities. Moreover, in the radial case, they showed that \eqref{uu1} holds with $\gamma\geq\frac{1}{4}$ for any $n\geq1$, and found the sharp constant $\C_{\text{rad},a}$ along with explicit optimizers. A connection between \eqref{uu1}, the Hardy inequality \eqref{kk1}, and the logarithmic Sobolev inequality was also explored in \cite{del1}. After the work \cite{del1}, to the best of our knowledge, there have been very few results concerning the general sharp constant $\C_a$ and the classification of extremizers. Here, we summarize several results.

\vskip0.1in

In \cite{Dol0}, Dolbeault and Esteban showed that for $n\geq 2$ and $\gamma>\frac{n}{4}$, the inequality \eqref{uu1} admits an extremal function. In the case $n\geq3$ and $\gamma=\frac{n}{4}$, there exists a constant $a'\in[-\infty,\frac{n-2}{2})$ such that when $a<a'$, the inequality \eqref{uu1} admits no extremal functions; and when $a>a'$, it admits an extremal function. An explicit upper bound of $a'$ was also obtained in \cite{Dol0}. In \cite{Dol2}, Dolbeault, Esteban, Tarantello and Tertikas showed that for $n\geq2$, there exists a continuous function $a^*:\;(\frac{n}{4},+\infty)\rightarrow(-\infty,\frac{n-2}{2})$ such that when $\gamma>\frac{n}{4}$, $\C_a=\C_{\text{rad},a}$ if and only if $a\geq a^*(\gamma)$. An explicit lower bound for $a^*$ was also derived. Later, in \cite{Dol0.5}, Dolbeault, Esteban, Filippas and Tertikas showed that when $n=2$ or $3$, $\gamma>\frac{3}{4}$, or $n\geq4$ and $\gamma\geq\frac{n}{4}$, there exists an explicit constant $a(n,\gamma)<\frac{n-2}{2}$ such that when $a\geq a(n,\gamma)$, any optimal function of \eqref{uu1} is radially symmetric.

\vskip0.1in

In the same spirit as Theorem \ref{thm1}, we can make some complements to the conclusions mentioned above. The notations $a'$, $a^*$ and $a(n,\gamma)$ are defined as above.
\begin{theorem}\label{thm3.1}
    When $n\geq 2$, $\gamma>\frac{n}{4}$, and $a^*(\gamma)<a<\frac{n-2}{2}$, any extremal function of \eqref{uu1} is radially symmetric. Moreover, a function $u\in D_a^2(\R^n)$ is a radial extremal function of \eqref{uu1} if and only $\Bar{u}\in D_{a^*(\gamma)}^2(\R^n)$ is a radial extremal function of \eqref{uu1} with $a=a^*(\gamma)$. Here, $\Bar{u}(|x|):=k^{\frac{1}{2}}u(|x|^k)$ with $k=\frac{n-2-2a^*(\gamma)}{n-2-2a}$.
\end{theorem}
\begin{theorem}\label{thm3.2}
    When $n\geq3$ and $\gamma=\frac{n}{4}$, there exists a constant $a_0\in[a',\frac{n-2}{2}]$ such that when $a'<a<a_0$, $\C_a<\C_{\text{rad},a}$, and when $a_0<a<\frac{n-2}{2}$, $\C_a=\C_{\text{rad},a}$ and any extremal function of \eqref{uu1} is radially symmetric. In particular, when $n\geq 4$, $a_0\leq a(n,\gamma)<\frac{n-2}{2}$.
\end{theorem}
\begin{proof}[Proof of the above two theorems]
    For Theorem \ref{thm3.1}, we use the transformation $u\rightarrow\Bar{u}(r\theta):=k^{\frac{1}{p}}u(r^k\theta)$ with $k=\frac{n-2-2a^*(\gamma)}{n-2-2a}$. For Theorem \ref{thm3.2}, we can assume there exists some $a<\frac{n-2}{2}$ such that $\C_a=\C_{\text{rad},a}$; otherwise, we can take $a_0=\frac{n-2}{2}$. Define
    \begin{align}
        a_0=\inf\left\{a\in \left(a',\frac{n-2}{2}\right)\;\Big|\;\C_a=\C_{\text{rad},a}\right\}.\nonumber
    \end{align}
    If $a_0>a'$, by continuity we know $\C_{a_0}=\C_{\text{rad},a_0}$. Now, for any $a_0<a<\frac{n-2}{2}$ and $u\in D_a^2(\R^n)$, we can use the transformation $u\rightarrow\Bar{u}(r\theta):=k^{\frac{1}{2}}u(r^k\theta)$ with $k=\frac{n-2-2a_0}{n-2-2a}$. If $a_0=a'$, for any $a_0<a<\frac{n-2}{2}$ and $u\in D_a^2(\R^n)$, we can choose $a_1\in(a',a)$ such that $\C_{a_1}=\C_{\text{rad},a_1}$. Then we can use the transformation $u\rightarrow\Bar{u}(r\theta):=k^{\frac{1}{2}}u(r^k\theta)$ with $k=\frac{n-2-2a_1}{n-2-2a}$. The remaining arguments are similar to those in the proof of Theorem \ref{thm1} and are thus omitted.
\end{proof}

\subsection{The Hardy-Sobolev-Morrey inequality}\label{sec3.4}
Recall the classical Hardy inequality \eqref{kk1}. The restriction $p<n$ is necessary for integrability near the origin. However, if we assume $u\in C_c^{\infty}(\R^n\backslash\{0\})$, then \eqref{kk1} can be extended to the case $p>n\geq1$:
\begin{align}\label{oo1}
    \int_{\R^n}|\nabla u|^pdx\geq\left(\frac{p-n}{p}\right)^p\int_{\R^n}\frac{|u|^p}{|x|^p}dx.
\end{align}
The constant $\left(\frac{p-n}{p}\right)^p$ is optimal but cannot be attained by nontrivial functions. Motivated by the pioneering work of Br\'ezis and V\'azquez \cite{Bre1}, as well as the classical Sobolev-Morrey inequality, Psaradakis \cite{Psa} improved the inequality \eqref{oo1} by adding a suitable H\"older norm of $u$. Specifically, the following two Hardy-Morrey-type inequalities have been established:

$(1)$ For any domain $0\in\Omega\subset\R^n$ with finite volume, $p>n\geq1$, and for every $u\in C_c^\infty(\Omega\backslash\{0\})$, there exists a positive constant $C=C(n,p)$ such that
\begin{align}\label{oo2}
     \left(\int_{\Omega}|\nabla u|^pdx-\left(\frac{p-n}{p}\right)^p\int_{\Omega}\frac{|u|^p}{|x|^p}dx\right)^{\frac{1}{p}}\geq C|\Omega|^{\frac{1}{p}-\frac{1}{n}}\sup_{x\in\Omega}|u(x)|.
\end{align}

$(2)$ For any bounded domain $0\in\Omega\subset\R^n$, $p>n\geq1$, and for every $u\in C_c^\infty(\Omega\backslash\{0\})$, there exist two positive constants $C=C(n,p)$ and $B=B(n,p)\geq 2$ such that
\begin{align}\label{oo3}
    \left(\int_{\Omega}|\nabla u|^pdx-\left(\frac{p-n}{p}\right)^p\int_{\Omega}\frac{|u|^p}{|x|^p}dx\right)^{\frac{1}{p}}\geq\;C\sup_{x\neq y}\left\{\frac{|u(x)-u(y)|}{|x-y|^{1-\frac{n}{p}}}X^{\frac{1}{p}}\left(\frac{|x-y|}{D}\right)\right\},
\end{align}
where $D:=B\sup\limits_{x\in\Omega} |x|$ and $X(t):=(1-\ln{t})^{-1}$. Moreover, the weight function $X^{\frac{1}{p}}$ is optimal in the sense that the exponent $\frac{1}{p}$ cannot be decreased.

By applying suitable transformations, we can establish the following Hardy-Sobolev-Morrey-type inequalities, which extend \eqref{oo2} and \eqref{oo3} to their weighted versions.
\begin{theorem}\label{thm4.1}
    For any domain $0\in\Omega\subset\R^n$ with finite volume, $p>n\geq1$, and $0<a<\frac{p-n}{p}$, there exist two positive constants $C_i=C_i(n,p,a)$ for $i=1,2$ such that for every $u\in C_c^\infty(\Omega\backslash\{0\})$, the following inequality holds:
\begin{align}\label{oo4}
     \left(\int_{\Omega}\frac{|\nabla u|^p}{|x|^{-ap}}dx-\left(\frac{p-n-pa}{p}\right)^p\int_{\Omega}\frac{|u|^p}{|x|^{p(1-a)}}dx\right)^{\frac{1}{p}}\geq&\; C_1|\Omega^{(p,a)}|^{\frac{1}{p}-\frac{1}{n}}\sup_{x\in\Omega}|u(x)|\nonumber\\
     \geq&\;C_2|\Omega|^{\frac{1}{p}-\frac{1-a}{n}}\sup_{x\in\Omega}|u(x)|,
\end{align}
where $\Omega^{(p,a)}:=\left\{x\;\big|\;|x|^{\frac{ap}{p-n-pa}}x\in\Omega\right\}$.
\end{theorem}
\begin{theorem}\label{thm4.2}
    For any bounded domain $0\in\Omega\subset\R^n$, with $p>n\geq1$ and $0<a<\frac{p-n}{p}$, there exist two positive constants $C=C(n,p,a)$ and $B=B(n,p,a)\geq 2$ such that for every $u\in C_c^\infty(\Omega\backslash\{0\})$, the following inequality holds:
\begin{align}\label{oo5}
     \left(\int_{\Omega}\frac{|\nabla u|^p}{|x|^{-ap}}dx-\left(\frac{p-n-pa}{p}\right)^p\int_{\Omega}\frac{|u|^p}{|x|^{p(1-a)}}dx\right)^{\frac{1}{p}}\geq C\sup_{x\neq y}\left\{\frac{|u(x)-u(y)|}{|x-y|^{1-a-\frac{n}{p}}}X^{\frac{1}{p}}\left(\frac{|x-y|}{D}\right)\right\},
\end{align}
where $D:=B\sup\limits_{x\in\Omega} |x|$ and $X(t):=(1-\ln{t})^{-1}$. Moreover, the exponent $\frac{1}{p}$ of $X$ cannot be decreased.
\end{theorem}
\begin{proof}[Proof of the above two theorems]
    Similar to the proofs of Theorem \ref{thm1} and Theorem \ref{thm2}, most of this proof can be obtained by using the transformation $u\rightarrow\Bar{u}(r\theta):=k^{\frac{1}{p}}u(r^k\theta)$ with $k=\frac{p-n}{p-n-pa}$ to reduce the weighted estimates \eqref{oo4} and \eqref{oo5} to their non-weighted versions, and then applying the results \eqref{oo2} and \eqref{oo3}. Here, we emphasize three points in this procedure.

    The first point is that among domains with fixed volume, the domain $\Omega$ for which $|\Omega^{(p,a)}|$ attains the maximum is precisely the ball centered at the origin. From this, one can easily deduce that $|\Omega^{(p,a)}|\leq C|\Omega|^{1-\frac{ap}{p-n}}$. (Note that this fact has been used in the proof of Theorem \ref{thm1.1}.)

    The second point is that to derive \eqref{oo5}, we need the following two simple assertions regarding the term $|x-y|$:

    $(1)$ For any positive constants $\alpha$ and $\beta$, the function $f(t):=t^{-\alpha}X^\beta(ct)$ is decreasing on $[0,1]$ when $0<c=c(\alpha,\beta)$ is sufficiently small.
    
    $(2)$ For any $0<s<1$ and $x,y\in\R^n\backslash\{0\}$, there exists a positive constant $C=C(s)$ such that
    \begin{align}\label{oo6}
        \left|\frac{x}{|x|^s}-\frac{y}{|y|^s}\right|\leq C|x-y|^{1-s}.
    \end{align}

    The first assertion $(1)$ follows from straightforward computation. As for the second assertion $(2)$, 
    by symmetry, homogeneity and rotational invariance, we can assume $n=2$, $x=(1,0)$, $y=(1+\varepsilon,\delta)$ and $|y|\leq 1$. Now, the inequality \eqref{oo6} is equivalent to
    \begin{align}\label{oo7}
        f(\varepsilon,\delta):=\left(1-\frac{1+\varepsilon}{\left((1+\varepsilon)^2+\delta^2\right)^\frac{s}{2}}\right)^2+\frac{\delta^2}{\left((1+\varepsilon)^2+\delta^2\right)^s}\leq C|\varepsilon|^{2-2s}+C|\delta|^{2-2s}.
    \end{align}
    Since $|y|\leq 1$, the left-hand side of \eqref{oo7} is bounded above, so we only need to consider the case when $|\varepsilon|+|\delta|$ is small. Note that $f$ is a nonnegative smooth function near $(0,0)$ and $f(0,0)=0$. Therefore, $f(\varepsilon,\delta)=\mathcal{O}(\varepsilon^2+\delta^2)$. The existence of the constant $C$ now follows directly. 

    The third point concerns the optimality of the exponent $\frac{1}{p}$ of $X$. From \cite[Section 7]{Psa}, we know that for any bounded domain $0\in\Omega_0\subset\R^n$, there exists a family of radially symmetric functions $v_\delta\in C_c^\infty(\Omega_0\backslash\{0\})$ for $0<\delta<1$ such that for any $B_0\geq 2$ and $0<\varepsilon<1$, it holds that
    \begin{align}\label{oo8}
        \frac{\left(\int_{\Omega_0}|\nabla v_\delta|^pdx-\left(\frac{p-n}{p}\right)^p\int_{\Omega_0}\frac{|v_\delta|^p}{|x|^p}dx\right)^{\frac{1}{p}}}{\sup\limits_{x\neq 0}\left\{\frac{|v_\delta(x)|}{|x|^{1-\frac{n}{p}}}X^{\frac{\varepsilon}{p}}\left(\frac{|x|}{D_0}\right)\right\}}\rightarrow0\quad\text{as }\delta\rightarrow0,
    \end{align}
    where $D_0:=B_0\sup\limits_{x\in\Omega_0} |x|$. Note that the optimality of the inequality \eqref{oo3} follows directly from \eqref{oo8}. In our case, we set $\Omega_0=\Omega^{(p,a)}$ in \eqref{oo8} and take a family of functions $u_\delta$ such that $\Bar{u}_\delta=v_\delta$. Due to the radial symmetry, direct computation shows that the estimate \eqref{oo8} is equivalent to
    \begin{align}\label{oo9}
        \frac{\left(\int_{\Omega}\frac{|\nabla u_\delta|^p}{|x|^{-ap}}dx-\left(\frac{p-n-pa}{p}\right)^p\int_{\Omega}\frac{|u_\delta|^p}{|x|^{p(1-a)}}dx\right)^{\frac{1}{p}}}{\sup_{x\neq 0}\left\{\frac{|u_\delta(x)|}{|x|^{1-a-\frac{n}{p}}}X^{\frac{\varepsilon}{p}}\left(\frac{|x|}{D}\right)\right\}}\rightarrow0\quad \text{as }\delta\rightarrow0.
    \end{align}
    The sharpness of the exponent $\frac{1}{p}$ follows directly from \eqref{oo9}.
\end{proof}

Our next results concerns the following estimate obtained by Gkikas and Psaradakis \cite{Gki}, which provides a series improvement of the inequality \eqref{oo3}. For any bounded domain $0\in\Omega\subset\R^n$, with $p>n\geq2$, $m\in\N$, and for every $u\in C_c^\infty(\Omega\backslash\{0\})$, there exist constants $C=C(n,p,m)>0$ and $B=B(n,p,m)\geq 2$ such that
\begin{align}\label{oo10}
    \Big(\int_{\Omega}|\nabla u|^pdx-\left(\frac{p-n}{p}\right)^p\int_{\Omega}\frac{|u|^p}{|x|^p}dx&-\frac{p-1}{2p}\left(\frac{p-n}{p}\right)^{p-2}\sum_{i=1}^m\int_{\Omega}\frac{|u|^p}{|x|^p}Y_i^2\left(\frac{|x|}{D}\right)dx\Big)^{\frac{1}{p}}\nonumber\\
    \geq&\; C\sup_{x\neq y}\left\{\frac{|u(x)-u(y)|}{|x-y|^{1-\frac{n}{p}}}Y_{m+1}^{\frac{1}{p}}\left(\frac{|x-y|}{D}\right)\right\},
\end{align}
where $D:=B\sup\limits_{x\in\Omega}|x|$, $Y_j(t):=\prod\limits_{i=1}^jX_j(t)$, $X_j(t):=X_1(X_{j-1}(t))$, and $X_1(t):=(1-\ln{t})^{-1}$. Moreover, the weight function $Y_{m+1}^{\frac{1}{p}}$ is optimal in the sense that the exponent $\frac{1}{p}$ on $X_{m+1}$ cannot be decreased.

Following similar arguments as in the proofs of Theorem \ref{thm1.4} and Theorem \ref{thm4.2}, we can establish a series improvement of the inequality \eqref{oo5} in the same spirit as \eqref{oo10}:
\begin{theorem}\label{thm4.3}
    For any bounded domain $0\in\Omega\subset\R^n$, with $p>n\geq2$, $0<a<\frac{p-n}{p}$, $m\in\N$, and for every $u\in C_c^\infty(\Omega\backslash\{0\})$, there exist constants $C=C(n,p,m,a)>0$ and $B=B(n,p,m,a)\geq 2$ such that
\begin{align}\label{oo11}
    \Big(&\;\int_{\Omega}\frac{|\nabla u|^p}{|x|^{-ap}}dx-\left(\frac{p-n-pa}{p}\right)^p\int_{\Omega}\frac{|u|^p}{|x|^{p(1-a)}}dx\nonumber\\
    &\;-\frac{p-1}{2p}\left(\frac{p-n-pa}{p}\right)^{p-2}\sum_{i=1}^m\int_{\Omega}\frac{|u|^p}{|x|^{p(1-a)}}Y_i^2\left(\frac{|x|}{D}\right)dx\Big)^{\frac{1}{p}}\nonumber\\
    \geq&\; C\sup_{x\neq y}\left\{\frac{|u(x)-u(y)|}{|x-y|^{1-a-\frac{n}{p}}}Y_{m+1}^{\frac{1}{p}}\left(\frac{|x-y|}{D}\right)\right\},
\end{align}
where $D$ and $Y_j$ are defined as above. Moreover, the weight function $Y_{m+1}^{\frac{1}{p}}$ is optimal.
\end{theorem}
\begin{proof}
    The inequality \eqref{oo11} can be derived from \eqref{oo10} using similar arguments as in the proofs of Theorem \ref{thm1.4} and Theorem \ref{thm4.2}. Regarding the sharpness of the exponent $\frac{1}{p}$ of $X_{m+1}$, we can follow the reasoning presented in Theorem \ref{thm1.4}. Specifically, if for some $0<\varepsilon<1$, the inequality \eqref{oo11} holds with the weight function $Y_m^{\frac{1}{p}}X_{m+1}^{\frac{\varepsilon}{p}}$, then there exists a number $\gamma<2$ such that the inequality
    \begin{align}
        &\;\int_{\Omega}\frac{|\nabla u|^p}{|x|^{-ap}}dx-\left(\frac{p-n-pa}{p}\right)^p\int_{\Omega}\frac{|u|^p}{|x|^{p(1-a)}}dx\nonumber\\
    &\;-\frac{p-1}{2p}\left(\frac{p-n-pa}{p}\right)^{p-2}\sum_{i=1}^m\int_{\Omega}\frac{|u|^p}{|x|^{p(1-a)}}Y_i^2\left(\frac{|x|}{D}\right)dx\nonumber\\
    \geq&\;C\int_{\Omega}\frac{|u|^p}{|x|^{p(1-a)}}Y_m^2\left(\frac{|x|}{D}\right)X_{m+1}^\gamma\left(\frac{|x|}{D}\right) dx\nonumber
    \end{align}
    holds for all $u\in C_c^\infty(\Omega\backslash\{0\})$. However, this leads to a contradiction with \cite[Theorem 2]{Barb0}.
\end{proof}

\subsection{The Hardy-Sobolev interpolation inequality}\label{sec3.5}
Based on the improved Hardy-type inequalities established in works such as \cite{Abd,Bre1}, one can  derive the following Hardy-Gagliardo-Nirenberg type interpolation inequality: for any $1<p<n$ and $p<q<\frac{np}{n-p}$, there exist a number $\gamma=\gamma(n,p,q)\in(0,1)$ and a positive constant $C=C(n,p,q)$ such that the inequality
\begin{align}\label{ii1}
    \left(\int_{\R^n}|\nabla u|^pdx-\left(\frac{n-p}{p}\right)^p\int_{\R^n}\frac{|u|^p}{|x|^p}dx\right)^{\frac{\gamma}{p}}\left(\int_{\R^n}|u|^pdx\right)^\frac{1-\gamma}{p}\geq C\left(\int_{\R^n}|u|^qdx\right)^\frac{1}{q}
\end{align}
holds for every $u\in W^{1,p}(\R^n)$. The restriction $q<\frac{np}{n-p}$ is necessary, as it is not possible to improve the sharp Hardy inequality \eqref{kk1} by adding the $L^\frac{np}{n-p}$-norm of $u$ (see, for example, \cite{Bre1,Fil}).

Recently, in \cite{Die}, by suitably modifying the terms in \eqref{ii1}, Dietze and Nam obtained an interesting Hardy-Sobolev interpolation inequality that allows for the critical Sobolev norm:  when  $p\geq 2$, $p<n\leq p+1$  and $\gamma\in\left[1-\frac{p}{n},\frac{1}{n}\right]$, then there exists a positive constant $C=C(n,p,\gamma)$ such that the inequality
\begin{align}\label{ii2}
     \left(\int_{\R^n}|\nabla u|^pdx-\left(\frac{n-p}{p}\right)^p\sup_{y\in\R^n}\int_{\R^n}\frac{|u(x)|^p}{|x-y|^p}dx\right)^{\gamma}\left(\sup_{y\in\R^n}\int_{\R^n}\frac{|u(x)|^p}{|x-y|^p}dx\right)^{1-\gamma}\geq C\norm*{u}_{L^\frac{np}{n-p}}^p
\end{align}
holds for any $u\in D_0^p(\R^n)$. Moreover, the lower bound of $\gamma$ is always sharp, and the upper bound is sharp when $p=2$ and $n=3$. It is worth mentioning that generally, one cannot remove the term $\sup\limits_{y\in\R^n}$ in \eqref{ii2}; otherwise, by translating a fixed function $u$ to infinity, a contradiction arises. However, in the radial setting, the inequality \eqref{ii2} can be enhanced. Specifically, from \cite{Die}, when $n\geq3$, $p=2$, $\gamma=\frac{1}{n}$, and $u\in D_{0,\text{rad}}^2(\R^n)$, it holds that
\begin{align}\label{ii3}
    \left(\int_{\R^n}|\nabla u|^2dx-\left(\frac{n-2}{2}\right)^2\int_{\R^n}\frac{|u(x)|^2}{|x|^2}dx\right)^{\gamma}\left(\int_{\R^n}\frac{|u(x)|^2}{|x|^2}dx\right)^{1-\gamma}\geq C_{\text{rad}}\norm*{u}_{L^\frac{2n}{n-2}}^2.
\end{align}
Moreover, the sharp constant $C_{\text{rad}}$ can be attained, and all the optimizers of \eqref{ii3} are given by the functions:
\begin{align}\label{ii4}
    u(x)=\alpha|x|^{-\frac{n-2}{2}}\left(\frac{|x|^\eta}{\beta+|x|^{2\eta}}\right)^\frac{n-2}{2},
\end{align}
where $\alpha\in\R$, $\beta>0$ and $\eta>0$. Additionally, $\frac{1}{n}$ is the only possible value of $\gamma$ such that the inequality \eqref{ii3} holds. A result closely related to \eqref{ii2} is the following improved Sobolev inequality involving Morrey norms obtained by Palatucci and Pisante \cite{Pal}: for $1\leq p<n$, $\gamma\in[1-\frac{p}{n},1)$, and $u\in D_0^p(\R^n)$, it holds that
\begin{align}\label{ii5}
    \left(\int_{\R^n}|\nabla u|^pdx\right)^{\gamma}\left(\sup_{R>0,x\in\R^n}R^{-p}\int_{B(x,R)}|u|^pdx\right)^{1-\gamma}\geq C\norm*{u}_{L^\frac{np}{n-p}}^p.
\end{align}
Note that a direct corollary of \eqref{ii5} is
\begin{align}\label{ii6}
    \left(\int_{\R^n}|\nabla u|^pdx\right)^{\gamma}\left(\sup_{y\in\R^n}\int_{\R^n}\frac{|u(x)|^p}{|x-y|^p}dx\right)^{1-\gamma}\geq C\norm*{u}_{L^\frac{np}{n-p}}^p,
\end{align}
which is a weaker version of \eqref{ii2} in some special cases. The inequality \eqref{ii6} plays an important role in the proof of \eqref{ii2}, and it provides the sharp lower bound $1-\frac{p}{n}$ for the exponent $\gamma$ in \eqref{ii2}.

\vskip0.2in
Here we aim to improve the Hardy-Sobolev inequality \eqref{kk3} in the same spirit as the inequalities \eqref{ii2} and \eqref{ii3}. The radial case is straightforward. However, for the general case, since both integrals of \eqref{kk3} have radial weights, it seems unnatural to introduce terms like $\sup\limits_{y\in\R^n}$. Additionally, there appear to be no weighted versions of the inequality \eqref{ii5} or \eqref{ii6}. Therefore, we only consider local improvements, namely, we assume that
\begin{align}\label{ii7}
    \int_{\R^n}|x|^{-ap}|\nabla u|^pdx\leq2\left(\frac{n-p-pa}{p}\right)^p\int_{\R^n}|x|^{-(a+1)p}|u|^pdx.
\end{align}
The number $2$ above is not crucial and can be replaced by any positive number larger than $1$. Our result is stated as follows.
\begin{theorem}\label{thm5.1}
    Assume $n>p\geq2$, $0\leq a<\frac{n-p}{p}$, and $\gamma\in[0,\frac{1}{n}]$. Then there exists a positive constant $C=C(n,p,a,\gamma)$ such that the inequality
    \begin{align}\label{ii8}
        \left(\int_{\R^n}\frac{|\nabla u|^p}{|x|^{ap}}dx-\left(\frac{n-p-pa}{p}\right)^p\int_{\R^n}\frac{|u|^p}{|x|^{p(a+1)}}dx\right)^{\gamma}\left(\int_{\R^n}\frac{|u|^p}{|x|^{p(a+1)}}dx\right)^{1-\gamma}\geq C\norm*{|x|^{-a}u}_{L^\frac{np}{n-p}}^p
    \end{align}
    holds for any $u\in D_a^p(\R^n)$ satisfying \eqref{ii7}. Moreover, when $p=2$, the upper bound $\frac{1}{n}$ for $\gamma$ is optimal.
\end{theorem}
\begin{theorem}\label{thm5.2}
    Assume $n\geq3$, $0<a<\frac{n-2}{2}$, $\gamma=\frac{1}{n}$, and $u\in D_{a,\text{rad}}^2(\R^n)$. Then it holds that
    \begin{align}\label{ii9}
        \left(\int_{\R^n}\frac{|\nabla u|^2}{|x|^{2a}}dx-\left(\frac{n-2-2a}{2}\right)^2\int_{\R^n}\frac{|u|^2}{|x|^{2(a+1)}}dx\right)^{\gamma}\left(\int_{\R^n}\frac{|u|^2}{|x|^{2(a+1)}}dx\right)^{1-\gamma}\geq C_{\text{rad}}\norm*{|x|^{-a}u}_{L^\frac{2n}{n-2}}^2,
    \end{align}
    where the sharp constant $C_{\text{rad}}$ is the same as that in \eqref{ii3}, and all the extremizers are given by the functions:
    \begin{align}
    u(x)=\alpha|x|^{-\frac{n-2-2a}{2}}\left(\frac{|x|^\eta}{\beta+|x|^{2\eta}}\right)^\frac{n-2}{2},\nonumber
\end{align}
where $\alpha\in\R$, $\beta>0$ and $\eta>0$. Moreover, the inequality \eqref{ii9} does not hold when $\gamma\neq\frac{1}{n}$.
\end{theorem}
\begin{proof}[Proof of the above two theorems]
    For Theorem \ref{thm5.2}, we consider the transformation $$u\rightarrow\Bar{u}(r\theta) :=k^{\frac{1}{2}}u(r^k\theta); \quad k=\frac{n-2}{n-2-2a}.$$  Direct computation shows that the estimate \eqref{ii9} is  indeed equivalent to \eqref{ii3}. This can also be verified using the ground state transformation $v(x):=|x|^{\frac{n-2-2a}{2}}u(x)$. For Theorem \ref{thm5.1}, we consider the transformation $u\rightarrow\Bar{u}(r\theta):=k^{\frac{1}{p}}u(r^k\theta)$ with $k=\frac{n-p}{n-p-pa}$. Note that the condition \eqref{ii7} is preserved under this transformation. Therefore, in the same spirit as Theorem \ref{thm1}, we only need to prove \eqref{ii8} for the case $a=0$. Additionally, from the assumption \eqref{ii7}, it suffices to consider the case when $\gamma=\frac{1}{n}$ and the ratio $\left(\frac{n-p}{p}\right)^{-p}\frac{\int_{\R^n}|\nabla u|^pdx}{\int_{\R^n}\frac{|u|^p}{|x|^{p}}dx}$ is close to $1$. Consider the function:
    \begin{align}
        f(t):=(1-t)^{\frac{1}{n}}t^{1-\frac{1}{n}},\quad t\in(0,1).\nonumber
    \end{align}
    It is straightforward to show that $f$ is strictly decreasing when $t\geq \frac{n-1}{n}$. Therefore, we obtain
    \begin{align}
         &\;\left(\int_{\R^n}|\nabla u|^pdx-\left(\frac{n-p}{p}\right)^p\int_{\R^n}\frac{|u(x)|^p}{|x|^p}dx\right)^{\frac{1}{n}}\left(\int_{\R^n}\frac{|u(x)|^p}{|x|^p}dx\right)^{1-\frac{1}{n}}\nonumber\\
         \geq&\;\left(\int_{\R^n}|\nabla u|^pdx-\left(\frac{n-p}{p}\right)^p\sup_{y\in\R^n}\int_{\R^n}\frac{|u(x)|^p}{|x-y|^p}dx\right)^{\frac{1}{n}}\left(\sup_{y\in\R^n}\int_{\R^n}\frac{|u(x)|^p}{|x-y|^p}dx\right)^{1-\frac{1}{n}}.\nonumber
    \end{align}
    Now, it suffices to show that the inequality \eqref{ii2} holds under the assumptions that $n>p\geq 2$, $\gamma=\frac{1}{n}$, and the ratio $\left(\frac{n-p}{p}\right)^{-p}\frac{\int_{\R^n}|\nabla u|^pdx}{\sup\limits_{y\in\R^n}\int_{\R^n}\frac{|u(x)|^p}{|x-y|^p}dx}$ is close to $1$. The proof of this is actually contained in \cite[Page 10, Case 2]{Die}. Finally, to establish the sharpness of the upper bound $\frac{1}{n}$ for $\gamma$ in the case $p=2$, one can refer to \cite[Formulas (18)-(20)]{Die}, or see our remark below. Thus, our proof is complete.
\end{proof}
\begin{remark}
    Consider the following test function for the inequality \eqref{ii8}:
    \begin{align}
        u(x):=\begin{cases}
            |x|^{\varepsilon-\frac{n-p-pa}{p}},\quad&\text{if }|x|\leq 1,\\
            |x|^{-\varepsilon-\frac{n-p-pa}{p}},\quad&\text{if }|x|\geq 1,
        \end{cases}\nonumber
    \end{align}
    where $0<\varepsilon<\frac{n-p-pa}{p}$. Direct computation yields:
    \begin{align}
    \norm*{|x|^{-a}u}_{L^\frac{np}{n-p}}^p\approx \varepsilon^{\frac{p}{n}-1},\quad\int_{\R^n}\frac{|u|^p}{|x|^{p(a+1)}}dx\approx \varepsilon^{-1},\nonumber
    \end{align}
    and
    \begin{align}
        \int_{\R^n}\frac{|\nabla u|^p}{|x|^{ap}}dx-\left(\frac{n-p-pa}{p}\right)^p\int_{\R^n}\frac{|u|^p}{|x|^{p(a+1)}}dx\approx \varepsilon.\nonumber
    \end{align}
    Thus, if \eqref{ii8} holds for the function $u$ defined above, we must have:
    \begin{align}
        \varepsilon^{\gamma}\varepsilon^{\gamma-1}\geq C\varepsilon^{\frac{p}{n}-1}\nonumber
    \end{align}
    for every $0<\varepsilon<\frac{n-p-pa}{p}$. This indicates that $\gamma\leq\frac{p}{2n}$. Therefore, when $p=2$, the upper bound $\frac{1}{n}$ for $\gamma$ is sharp. For the case $p>2$, the necessary upper bound $\frac{p}{2n}$ for $\gamma$ is strictly larger than the number $\frac{1}{n}$ that we can prove. Whether $\frac{1}{n}$ is also sharp for $\gamma$ in this case, or if \eqref{ii8} holds for $\gamma\in[\frac{1}{n},\frac{p}{2n}]$, remains an open question.
\end{remark}

\subsection{The interpolated Caffarelli-Kohn-Nirenberg inequality}\label{sec3.6}
In this subsection, we consider the following interpolated Caffarelli-Kohn-Nirenberg inequality:
\begin{align}\label{pp1}
    \left(\int_{\R^n}\frac{|\nabla u|^2}{|x|^{2b}}dx\right)^\frac{1}{2}\left(\int_{\R^n}\frac{|u|^2}{|x|^{2a}}dx\right)^\frac{1}{2}\geq \C_{n,a,b}\int_{\R^n}\frac{|u|^2}{|x|^{a+b+1}}dx,
\end{align}
where $n\geq 3$, $a,b\in\R$, and $u\in C_c^\infty(\R^n\backslash\{0\})$. For the explicit sharp constant $\C_{n,a,b}$ and the classification of extremal functions, we refer to the works \cite{Cat,Caz,Cos}. When $a=b+1$, this inequality reduces to the Hardy-Sobolev inequality \eqref{kk3}. When $a=-1$ and $b=0$, it reduces to the classical Heisenberg-Pauli-Weyl inequality, where $\C_{n,-1,0}=\frac{n}{2}$ and all optimizers are given by standard Gaussian functions of the form $\lambda e^{-\alpha|x|^2}$, where $\lambda\in\R$ and $\alpha>0$. When $a=b=0$, the inequality corresponds to the hydrogen uncertainty principle, which is closely related to the ground states of hydrogenic atoms with Coulomb interaction.

Recently, by establishing various Caffarelli-Kohn-Nirenberg-type and Hardy-type identities, and utilizing suitable Poincar\'e-type inequalities, Cazacu, Flynn, Lam and Lu \cite{Caz1} obtained various stability estimates for the inequality \eqref{pp1} under the assumptions:
\begin{align}\label{pp2}
    n\geq3,\quad0\leq b<\frac{n-2}{2},\quad a=\frac{n+2}{n-2}b-1.
\end{align}
Specifically, from the identities they first easily recovered the results in \cite{Cat,Caz,Cos}: the sharp constant $\C_{n,a,b}=\frac{n-a-b-1}{2}$ and all optimizers are given by Gaussians of the form $\lambda e^{-\alpha|x|^{b+1-a}}$, where $\lambda\in\R$ and $\alpha>0$. Then, by applying suitable Poincar\'e-type inequalities to the identities, they obtained the following stability result: there exist positive constants $C_i=C_i(n,a,b)$ for $i=1,2$ such that
\begin{align}\label{pp3}
   &\; \left(\int_{\R^n}\frac{|\nabla u|^2}{|x|^{2b}}dx\right)^\frac{1}{2}\left(\int_{\R^n}\frac{|u|^2}{|x|^{2a}}dx\right)^\frac{1}{2}- \frac{n-a-b-1}{2}\int_{\R^n}\frac{|u|^2}{|x|^{a+b+1}}dx\nonumber\\
   \geq&\;\begin{cases}
       C_1\inf\limits_{\lambda\in\R,\alpha>0}\int_{\R^n}\frac{\left|u-\lambda e^{-\alpha|x|^{b+1-a}}\right|^2}{|x|^{a+b+1}}dx,\\
       C_2\inf\limits_{\lambda\neq0}\int_{\R^n}\frac{\left|u-\lambda e^{-\alpha_\lambda|x|^{b+1-a}}\right|^2}{|x|^{a+b+1}}dx,
   \end{cases}
\end{align}
where $\alpha_\lambda$ is a positive constant that depends on $\lambda$ such that
\begin{align}\label{pp4}
    \int_{\R^n}\frac{|u|^2}{|x|^{a+b+1}}dx=\int_{\R^n}\frac{\left|\lambda e^{-\alpha_\lambda|x|^{b+1-a}}\right|^2}{|x|^{a+b+1}}dx.
\end{align}
In the particular case $a=-1$ and $b=0$, they showed that the sharp stability constants are $C_1=1$ and $C_2=\frac{1}{2}$. Moreover, the equalities in \eqref{pp3} can be achieved by nontrivial functions.

\vskip0.12in
Besides the stability estimate \eqref{pp3} for the scale-invariant inequality \eqref{pp1}, they also derived the following scale non-invariant stability result: there exists a positive constant $C_3=C_3(n,a,b)$ such that
\begin{align}\label{pp5}
    &\;\int_{\R^n}\frac{|\nabla u|^2}{|x|^{2b}}dx+\int_{\R^n}\frac{|u|^2}{|x|^{2a}}-(n-a-b-1)\int_{\R^n}\frac{|u|^2}{|x|^{a+b+1}}dx\nonumber\\
    \geq&\;C_3\inf_{\lambda\neq 0}\left[\begin{aligned}
        \int_{\R^n}&\frac{\left|u-\lambda e^{-\frac{|x|^{b+1-a}}{b+1-a}}\right|^2}{|x|^{a+b+1}}dx+\int_{\R^n}\frac{\left|u-\lambda e^{-\frac{|x|^{b+1-a}}{b+1-a}}\right|^2}{|x|^{2a}}dx\\
        &+\int_{\R^n}\frac{\left|\nabla \left(u-\lambda e^{-\frac{|x|^{b+1-a}}{b+1-a}}\right)\right|^2}{|x|^{2b}}dx.
    \end{aligned}\right]
\end{align}
In the particular case $a=-1$ and $b=0$, they showed that the sharp stability constant $C_3=\frac{2}{n+3}$ can be achieved by nontrivial functions.

\vskip0.12in

In our opinion, the main reason that the explicit (and sharp) stability constants $C_1,C_2,C_3$ were obtained only in the particular case $a=-1$ and $b=0$ is that the explicit (and sharp) constant for the Poincar\'e-type inequalities they used is only known in this case.

 \vskip0.12in

 Here, by employing a suitable transformation in the same spirit of Theorem \ref{thm2} and Theorem \ref{thm4}, we not only provide a simple recovery of the inequalities \eqref{pp3} and \eqref{pp5}, but also offer explicit stability constants that are straightforward in expression.
\begin{theorem}\label{thm6.1}
    Assume that \eqref{pp2} holds and $u\in C_c^\infty(\R^n)$, then the inequality \eqref{pp3} holds with the constants $C_1=\frac{b+1-a}{2}$ and $C_2=\frac{b+1-a}{4}$. Additionally, the inequality \eqref{pp5} holds with the constant $C_3=\frac{2}{n+3}\left(\frac{b+1-a}{2}\right)^2$.
\end{theorem}
\begin{proof}
    For simplicity, we denote the left-hand side of \eqref{pp3} as $\delta_{a,b}(u)$ and the left-hand side of \eqref{pp5} as $\Bar{\delta}_{a,b}(u)$. Additionally, we use $d^{(1)}_{a,b}(u)$, $d^{(2)}_{a,b}(u)$, and $d^{(3)}_{a,b}(u)$ to represent the three infima that appear on the right-hand side of \eqref{pp3} and \eqref{pp5}, respectively.

    In the case $b=0$ and $a=-1$, the explicit constants $C_1$, $C_2$ and $C_3$ have been obtained in \cite{Caz1}. For the case $b>0$, we consider the transformation $u\rightarrow\Bar{u}(r\theta) :=k^{-\frac{n}{4}}u(k^{-\frac{k}{2}}r^k\theta)$ with $k=\frac{2}{b+1-a}$. As in the equations \eqref{for1}-\eqref{for5}, we can compute the following:
    \begin{align}
         \left(\frac{2}{b+1-a}\right)^2\int_{\R^n}|\nabla \Bar{u}|^2dx\geq\int_{\R^n}\frac{|\nabla u|^2}{|x|^{2b}}dx\geq \int_{\R^n}|\nabla \Bar{u}|^2dx,\nonumber
    \end{align}
    \begin{align}
        \int_{\R^n}\frac{|u|^2}{|x|^{2a}}=\int_{\R^n}|x|^2|\Bar{u}|^2dx,\nonumber
    \end{align}
    and
    \begin{align}
        \int_{\R^n}\frac{|u|^2}{|x|^{a+b+1}}dx=\frac{2}{b+1-a}\int_{\R^n}|\Bar{u}|^2dx.\nonumber
    \end{align}
    Now, following similar  approach as in the proofs of Theorem \ref{thm2} and Theorem \ref{thm4}, we can deduce that
    \begin{align}\label{pp6}
        \delta_{a,b}(u)\geq \delta_{-1,0}(\Bar{u})\geq \begin{cases}
            d^{(1)}_{-1,0}(\Bar{u})\\
            \frac{1}{2}d^{(2)}_{-1,0}(\Bar{u})
        \end{cases}= \begin{cases}
            \frac{b+1-a}{2}d^{(1)}_{a,b}(u)\\
            \frac{b+1-a}{4}d^{(2)}_{a,b}(u),
        \end{cases}
    \end{align}
    and
    \begin{align}\label{pp7}
        &\;\Bar{\delta}_{a,b}(u)\geq \Bar{\delta}_{-1,0}(\Bar{u})\geq \frac{2}{n+3}d^{(3)}_{-1,0}(\Bar{u})\nonumber\\
        \geq&\;\frac{2}{n+3}\left[\begin{aligned}
        \frac{b+1-a}{2}&\int_{\R^n}\frac{\left|u-\lambda e^{-\frac{|x|^{b+1-a}}{b+1-a}}\right|^2}{|x|^{a+b+1}}dx+\int_{\R^n}\frac{\left|u-\lambda e^{-\frac{|x|^{b+1-a}}{b+1-a}}\right|^2}{|x|^{2a}}dx\\
        +&\left(\frac{b+1-a}{2}\right)^2\int_{\R^n}\frac{\left|\nabla \left(u-\lambda e^{-\frac{|x|^{b+1-a}}{b+1-a}}\right)\right|^2}{|x|^{2b}}dx.
    \end{aligned}\right]\nonumber\\
        \geq&\; \frac{2}{n+3}\left(\frac{b+1-a}{2}\right)^2d^{(3)}_{a,b}(u).
    \end{align}
    The proof is complete.
\end{proof}
\begin{remark}
    For the case $b>0$, although the explicit stability constants we find have simple forms, we suspect that they are not sharp. Indeed, from the proof above, we can show that these constants cannot be attained by nontrivial functions. Specifically, if $C_1$ or $C_2$ can be achieved by some nontrivial function $u$, then all inequalities in \eqref{pp6}  must hold with equality. From the first equality, we deduce that both $u$ and $\Bar{u}$ must be radially symmetric. For the second equality, by reviewing the arguments presented in \cite[Theorem 3.2]{Caz1}, we find that the following sharp Poincar\'e inequality
    \begin{align}
        \frac{\lambda^2}{2}\int_{\R^n}\left|\nabla v\right|^2e^{-\frac{|x|^2}{\lambda^2}}dx\geq \inf_{c\in\R}\int_{\R^n}\left|v-c\right|^2e^{-\frac{|x|^2}{\lambda^2}}dx\nonumber
    \end{align}
    must take equality when $v=\Bar{u}e^{\frac{|x|^2}{2\lambda^2}}$ and $\lambda=\left(\frac{\int_{\R^n}|x|^2|\Bar{u}|^2dx}{\int_{\R^n}|\nabla \Bar{u}|^2dx}\right)^\frac{1}{4}$. Since $v$ is radially symmetric, the Gaussian rearrangement inequality and Gaussian P\'olya-Szeg\"o inequality imply that $v$ must be a constant. This leads us to conclude that $u$ is an extremal function for the inequality \eqref{pp1}, which contradicts our assumption that $u$ is nontrivial. As for $C_3$, it is straightforward to see that the inequalities in \eqref{pp7} cannot hold with equality simultaneously unless $u$ is an extremal function for \eqref{pp1}.
\end{remark}

Finally, in the same spirit as \cite[Theorem 1.3]{Caz1}, we can obtain the following two estimates by taking the square of \eqref{pp3}:
\begin{align}
    &\;\int_{\R^n}\frac{|\nabla u|^2}{|x|^{2b}}dx\int_{\R^n}\frac{|u|^2}{|x|^{2a}}dx-\left(\frac{n-a-b-1}{2}\right)^2\left(\int_{\R^n}\frac{|u|^2}{|x|^{a+b+1}}dx\right)^2\nonumber\\
    \geq&\;(n-a-b-1)\frac{b+1-a}{2}\left(\int_{\R^n}\frac{|u|^2}{|x|^{a+b+1}}dx\right)d^{(1)}_{a,b}(u)+\left(\frac{b+1-a}{2}\right)^2d^{(1)}_{a,b}(u)^2\nonumber
\end{align}
and
\begin{align}
    &\;\int_{\R^n}\frac{|\nabla u|^2}{|x|^{2b}}dx\int_{\R^n}\frac{|u|^2}{|x|^{2a}}dx-\left(\frac{n-a-b-1}{2}\right)^2\left(\int_{\R^n}\frac{|u|^2}{|x|^{a+b+1}}dx\right)^2\nonumber\\
    \geq&\;(n-a-b-1)\frac{b+1-a}{4}\left(\int_{\R^n}\frac{|u|^2}{|x|^{a+b+1}}dx\right)d^{(2)}_{a,b}(u)+\left(\frac{b+1-a}{4}\right)^2d^{(2)}_{a,b}(u)^2.\nonumber
\end{align}
Here, $d^{(1)}_{a,b}$ and $d^{(2)}_{a,b}$ are defined as in the proof of Theorem \ref{thm6.1}.

\vspace{30pt}
\addcontentsline{toc}{section}{Declarations
}
\noindent{\Large\textbf{Declarations}}
~\\
~\\
\textbf{Conflict of interest}\quad On behalf of all authors, the corresponding author states that there is no Conflict of interest.~\\
{\textbf{Data Availability Statements}}\quad All data generated or analyzed during this study are included in this article.

\end{document}